\documentclass[a4paper,12pt]{article}
\usepackage{graphics}
\usepackage{epsf}

\usepackage{amsfonts}
\usepackage{amssymb}
\usepackage{epsfig}
\usepackage[latin1]{inputenc}
\usepackage{color}

\usepackage{amsthm,amsmath,amssymb}

\newtheorem{theorem}{Theorem}

\newtheorem{lemma}[theorem]{Lemma}
\newtheorem{remark}[theorem]{Remark}

\begin{document}

\title{Avoiding order reduction when integrating diffusion-reaction boundary value problems with exponential splitting methods}
\author{{\sc I. Alonso-Mallo \thanks{Email: isaias@mac.uva.es}, B. Cano \thanks{Corresponding author. Email: bego@mac.uva.es} } \\ \small
IMUVA, Departamento de Matem\'atica Aplicada,\\ \small Facultad de
Ciencias, Universidad de
Valladolid,\\ \small Paseo de Bel\'en 7, 47011 Valladolid,\\ \small Spain \\
{\sc and}\\
{\sc N. Reguera}\thanks{Email: nreguera@ubu.es} \\
\small IMUVA, Departamento de Matem\'aticas y Computaci\'on, \\
\small  Escuela Polit\'ecnica Superior, Universidad de Burgos,\\
\small  Avda. Cantabria, 09006 Burgos, \\ \small  Spain }
\date{}

%
%

\maketitle

\begin{abstract}
In this paper, we suggest a technique to avoid order reduction in
time when integrating reaction-diffusion boundary
value problems under non-homogeneous boundary conditions with
exponential splitting methods. More precisely, we consider
Lie-Trotter and Strang splitting methods and Dirichlet, Neumann
and Robin boundary conditions. Beginning from an abstract
framework in Banach spaces, a thorough error analysis after full
discretization is performed and some numerical results are shown
which corroborate the theoretical results.
\end{abstract}



\section{Introduction}

Exponential splitting methods are very much used in the recent
literature when integrating partial differential equations because
they integrate the linear and stiff part of the problem in an
exact way \cite{HO2}. Due to the recent and high
development of Krylov-type methods to calculate exponential-type
functions over matrices which are applied over vectors
\cite{Grimm}, they constitute an effective tool to integrate such
problems in a stable way.

In this paper, we will center on the first-order Lie-Trotter and
second-order Strang methods. The order reduction which turns up
with these methods when integrating linear problems with
homogeneous boundary conditions was recently
studied in \cite{FOS}. In \cite{ACR} we have suggested a technique
to deal with non-homogeneous boundary conditions in linear
problems. That technique has some similarities to that suggested
in \cite{acr2} for other exponential-type methods, which are
Lawson ones. With that procedure, we  managed to avoid order
reduction completely in linear problems.

The aim of  the present paper is to generalize that technique to
nonlinear reaction-diffusion problems and to prove that order reduction can also be completely avoided.

There are other results in the literature concerning this problem
or a more specific one. For example, in \cite{CR}, a generalized
Strang method is suggested for the specific nonlinear
Schr\"odinger equation. However, in that paper, an abstract
formulation of the problem is not given (as it is here), Neumann
or Robin type boundary conditions are not considered,  parabolic problems for which a
summation-by-parts argument can be applied are not included and
finally, Lie-Trotter method is not analyzed. On the other hand, in
\cite{EO, EO2}, a completely different technique is suggested to avoid
order reduction with the same methods and nonlinear problems than
here, but the analysis for the local and global error is just
performed in time. The error coming from the space discretization and the numerical approximation in time of the nonlinear and smooth part
are not included  and therefore, a practical implementation of the
technique for the practitioners is not justified. However, in the
present paper, the exact formulas to be implemented are described
in (\ref{ffv})-(\ref{Uhn_1}) for Lie-Trotter and in (\ref{Vhnsf}),
(\ref{ffws}) and (\ref{Uhn_1s}) for Strang. Moreover, the analysis is performed under quite general assumptions on the space discretization and time integration of the nonlinear part. For that, we use the maximum norm, which facilitates its applicability to quite general problems.

The paper is structured as follows. Section 2 gives some
preliminaries on the abstract setting of the problem, on the
assumptions of regularity which are required for the solution to
be approximated and on Lie-Trotter and Strang methods. Section 3
describes the technique to avoid order reduction after time
integration with Lie-Trotter method and explains how to deal with
non-homogeneous Dirichlet, Robin and Neumann type boundary
conditions. Moreover, a thorough local error analysis is given. In
Section 4, the same is done for Strang method, for which just
order $2$ can be obtained in general for the local error. Section
5 states some hypotheses on the space discretization which include some
finite-difference schemes, as the ones being used in the numerical experiments. (Similarly, collocation-type methods could be considered.) The detailed
analysis of the local and global error after full discretization
is performed in Section 6 for Lie-Trotter splitting and in Section
7 for Strang splitting. For the latter, no order reduction in time
is observed for the global error if the bound (\ref{spp}) is
satisfied by the discretization of the elliptic problem. Finally,
Section 8 shows some numerical experiments which corroborate the
previous results. Moreover, in two dimensions a double splitting
is included considering the results in \cite{ACR}.

\section{Preliminaries}
Let $X$ and $Y$ be Banach spaces and let $A:D(A) \to X$ and
$\partial: X \to Y$ be linear operators. Our goal is to study full
discretizations, by using as time integrators Lie-Trotter and
Strang exponential methods, of the nonlinear abstract non homogeneous initial
boundary value problem
\begin{eqnarray}
\label{laibvp}
\begin{array}{rcl}
u'(t)&=&Au(t)+f(t,u(t)), \quad  0\le t \le T,\\
u(0)&=&u_0 \in X,\\
\partial u(t)&=&g(t)\in Y, \quad  0\le t \le T,
\end{array}
\end{eqnarray}
where the functions $f:[0,T]\times X \to X$ (in general nonlinear)
and $g: [0,T] \to Y$ are  regular enough.

The abstract setting (\ref{laibvp}) permits to cover a wide range
of nonlinear evolutionary problems governed by partial
differential equations. We use the following hypotheses, which are
closely related to the ones in \cite{hansenko}, where the Strang
splitting  applied to a similar abstract problem with homogeneous
boundary conditions is studied. In our case, we add suitable
hypotheses in a such way that we are able to consider non
homogeneous boundary values (cf. \cite{alonsomallop,palenciaa}).

\begin{enumerate}
\item[(A1)] The boundary operator $\partial:D(A)\subset X\to Y$ is
onto.

\item[(A2)] Ker($\partial$) is dense in $X$ and
$A_0:D(A_0)=\ker(\partial)\subset X \to X$, the restriction of $A$
to Ker($\partial$), is the infinitesimal generator of a $C_0$-
semigroup $\{e^{t A_0}\}_{t\ge 0}$ in $X$, which type $\omega$ is
assumed to be negative.

\item[(A3)] If $z \in \mathbb{C}$ satisfies $\Re (z) >0$ and $v
\in Y$, then the steady state problem
\begin{eqnarray}
 Ax &=& zx,\\
 \partial x&=&v,
\label{stationaryproblem}
\end{eqnarray}
possesses a unique solution denoted by $x=K(z)v$. Moreover, the
linear operator $K(z): Y \to D(A)$ satisfies
\begin{eqnarray}
\label{stationaryoperator} \| K(z)v\| \le  C\|v\|,
\end{eqnarray}
where the constant $C$ holds for any $z$ such that $Re (z) \ge
\omega_0 > \omega$.

\item[(A4)] The nonlinear source $f$ belongs to $C^1([0,T] \times
X, X)$.

\item[(A5)] The solution $u$ of (\ref{laibvp}) satisfies $u\in
C^2([0,T], X)$, $u(t) \in D(A^2)$ for all $t \in [0,T]$ and $Au,
A^2u \in C^{1}([0,T], X)$.

\item[(A6)] $f(t, u(t))\in D(A)$ for  all $t \in [0,T]$, and
$Af(\cdot, u(\cdot))\in C([0,T], X)$.
\end{enumerate}

In the remaining of the paper, we always suppose that (A1)-(A6)
are satisfied. However, we notice that we also assume more
regularity in certain results which apply for Strang method.

\begin{remark}
From (A4), we deduce that $f :D(f)\subset [0,T] \times X \to X$ is
locally Lipschitz continuous in $u$, uniformly in $t \in [0,T]$,
with respect to the norm in $X$, that is,
\begin{eqnarray}
\label{locallipschtz}
  \|f(t, v)-f(t,u)\| \le  L(c)\|v-u \|,
\end{eqnarray}
for $(t,u), (t, v) \in D(f)$ with $\|u\|, \|v\| \le c$.

In order to define the Lie-Trotter and Strang splitting methods,
we need to solve the nonlinear evolution equation
\begin{eqnarray}
\label{fivp}
\begin{array}{rcl}
v'(t)&=& f(\tau+t, v(t)),\\
v(0)&=&v_0,
\end{array}
\end{eqnarray}
for several initial values $v_0$ and times $\tau>0$.  From
(\ref{locallipschtz}), problem (\ref{fivp}) has a unique
solution, which is well defined for sufficiently small times (see
 Theorem 1.8.1 in \cite{cartan}).
\end{remark}

\begin{remark}
When problem (\ref{laibvp}) is linear, that is, when
$f(t,\cdot) \equiv h(t)$, the results in \cite{
alonsomallop,palenciaa} show that, with the hypotheses (A1)-(A3),
the problem
\begin{eqnarray}
\label{laibvp2}
\begin{array}{rcl}
u'(t)&=&Au(t)+h(t), \quad  0\le t \le T,\\
u(0)&=&u_0 \in X,\\
\partial u(t)&=&g(t)\in Y, \quad  0\le t \le T,
\end{array}
\end{eqnarray}
is well posed and the solution depends continuously on data $u_0,
h$, and $g$.

In order to define the time integrators which are used in this
paper, we will consider initial boundary value problems which can
be written as
\begin{eqnarray}
\label{auxprob} \left.\begin{array}{rcl}
u'(s)&=&A u(s),\\
u(0)&=&u_0,\\
\partial u(s)&=& v_0 +v_1s,
\end{array}\right.
\end{eqnarray}
where $u_0\in X$ and $v_0, v_1\in Y$.

Assuming that $u_0 \in D(A)$ and $\partial u_0 =v_0$,  the
solution of (\ref{auxprob}) is given by (see e.g. \cite{ACR})
\begin{eqnarray}
\label{auxsol} u(t)= e^{t A_0}\left(u_0-K(0)v_0\right) + K(0)(v_0
+v_1t) -\int_0^t e^{s A_0}K(0)v_1ds.
\end{eqnarray}

Notice that (\ref{auxsol}) is well defined for any $u_0 \in X$ and
$v_0, v_1\in Y$; therefore, it may be considered as a generalized
solution of (\ref{auxprob}) even when $\partial u_0 \not= v_0$ or
$u _0 \notin D(A)$. We will use this fact in order to establish
the time integrator method in the following section.

\end{remark}

\begin{remark}
From hypotheses (A1)-(A4),  problem (\ref{laibvp}) with
homogeneous boundary conditions has a unique classical solution
for small enough time intervals (see Theorem 6.1.5 in
\cite{pazy}).

Regarding the nonhomogeneous case, we can assume that the boundary
function $g:[0,T] \to Y$ satisfies $g \in C^1([0,T],Y)$ and we can
look for a solution of (\ref{laibvp}) given by:
\begin{eqnarray*}
u(t)=v(t)+K(z)g(t), \quad t \ge 0,
\end{eqnarray*}
for some fixed $\Re(z) >\omega$. Then, $v$ is solution of an IBVP
with vanishing boundary values similar to the one in \cite{pazy}
and the well-posedness  for the case of
nonhomogeneous boundary values is a direct consequence if we take
the abstract theory for initial boundary value problems in
\cite{alonsomallop, palenciaa} into account.

However, condition (A4) may be very strong. When  $X$ is a
function space with a norm  $L^p$, $ 1 \le p < +\infty$, and $f$
is an operator given by,
\begin{eqnarray}
u \to f(u)= \phi \circ u, \label{neminski}
\end{eqnarray}
with $\phi : \mathbb{C} \to \mathbb{C}$, (\ref{locallipschtz})
implies that $\phi$ is globally Lipschitz in $\mathbb{C}$. This
objection disappears by considering the supremum norm, which is
used in our numerical examples, where the nonlinear source is
given by
\begin{eqnarray}
\label{nonlinearterm} u \to f(t,u)= \phi \circ u + h(t),
\end{eqnarray}
with $h:[0,T] \to X$, that is, $f$ is the sum of an
operator like (\ref{neminski}) and a linear term. In this way, problem (\ref{laibvp}) is
well posed.
\end{remark}

\begin{remark}
If we suppose that $A_0$ is the infinitesimal generator of an
analytic semigroup, we can consider, for $\theta \in (0,1)$ a new
norm given by $\| u\|_\theta=\| (\omega I-A_0)^\theta u\|$,
$\omega >0$, when $u \in X_\theta=D((\omega I-A_0)^\theta)$. In
this case, if $f$ satisfies a local Lipschitz condition in $u$
with this new norm, it is possible to obtain the well posedness of problem (\ref{laibvp}) even when $X$ is a function space with
a norm $L^p$, $ 1 \le p < +\infty$ (see \cite{henry}). However,
this approach is not enough for our purposes since we also need to
solve the nonlinear evolution equation (\ref{fivp}).
\end{remark}

{\bf Example.}
 {\it Let  $\Omega \subset \mathbb{R}^n$ be open and
bounded with Lipschitz boundary. Then, there exists a unique
solution $h \in C(\overline{\Omega})$ of the problem
\begin{eqnarray}
\label{estproblem}
\begin{array}{rcl}
\Delta h&=& 0\,\,\mathrm{in} \,D(\Omega)', \\
h|_{\partial \Omega}&=&\varphi \in C(\partial \Omega),
\end{array}
\end{eqnarray}
where $D(\Omega)'$ denotes the space of distributions. We also
remark that it can be proved that $h \in C^\infty(\Omega)$.

We take $X=C(\overline{\Omega})$  with the supremum norm and we
consider the operator $A_0$ defined on $X$ by
\begin{eqnarray*}
D(A_0)&=&\{ u\in C_0(\Omega): \Delta u \in X\}\\
A_0u &=& \Delta u \in D(\Omega)',
\end{eqnarray*}
where $C_0(\Omega)=\{ u \in X :  u|_{\partial \Omega} =0\}$. Then,
the operator $A_0$ generates a bounded holomorphic semigroup
$e^{tA_0}$ on X (\cite{arendt}, Section 2.4). We denote $\omega
<0$ the type of this semigroup.

Now, we take $Y=C(\partial \Omega)$ and we define the linear
operator
\begin{eqnarray*}
K: Y &\to& K(Y) \subset X\\
\varphi &\to& K(\varphi)=h,
\end{eqnarray*}
where $h$ is the solution of (\ref{estproblem}). Then, we can
define the (dense) subspace
\begin{eqnarray*}
D(A)=  D(A_0)\oplus K(Y),
\end{eqnarray*}
the extension of the operator $A_0$,
\begin{eqnarray*}
A: D(A)\subset X &\to& X
\end{eqnarray*}
by means of $Au=\Delta u$  for each $u \in D(A)$, and the boundary
operator
\begin{eqnarray*}
\partial : D(A) \subset X &\to& Y,\\
 u &\to& \partial u=u|_{\partial \Omega}.
\end{eqnarray*}

Finally, if $ z \in \mathbb{C}$ satisfies $ \Re (z) > \omega$, we
define
\begin{eqnarray*}
K(z)= (-A_0) (z-A_0)^{-1}K= K -z(z-A_0)^{-1}K,
\end{eqnarray*}
which satisfies $AK(z)= AK -zA_0(z-A_0)^{-1}K=
-zA_0(z-A_0)^{-1}K=zK(z)$.

Therefore, hypotheses (A1),(A2), and (A3) are satisfied.

We remark that the restriction to $A=\Delta$ is only made for
simplicity of presentation and more general elliptic operators can
be considered (see \cite{arendt2}).
}

Because of hypothesis (A2), $\{\varphi_j(t A_0)\}_{j=1}^{3}$   are
bounded operators for $t>0$, where $\{\varphi_j\}$ are the
standard functions which are used in exponential methods \cite{HO2} and which
are defined by
\begin{eqnarray}
\varphi_j( t A_0)=\frac{1}{t^j} \int_0^t
e^{(t-\tau)A_0}\frac{\tau^{j-1}}{(j-1)!}d\tau, \quad j \ge 1.
\label{varphi}
\end{eqnarray}
It is well-known that they can be calculated in a recursive way through the formulas
\begin{eqnarray}
\varphi_{j+1}(z)=\frac{\varphi_j(z)-1/j!}{z}, \quad z \neq 0,
\qquad \varphi_{j+1}(0)=\frac{1}{(j+1)!}, \qquad \varphi_0(z)=e^z.
\label{recurf}
\end{eqnarray}

For the time integration, we will center on exponential
Lie-Trotter and Strang methods which, applied to a
finite-dimensional nonlinear problem like
\begin{eqnarray}
U'(t) = M U(t)+F(t,U(t)), \label{linfd}
\end{eqnarray}
where $M$ is a matrix, are described by the following formulas at
each step
\begin{eqnarray}
U_{n+1}&=&\Psi_{k}^{F,t_n}( e^{ k M} U_n), \label{lie} \\
U_{n+1}&=&\Psi_{\frac{k}{2}}^{F,t_n+\frac{k}{2}} \big( e^{ k M}
\Psi_{\frac{k}{2}}^{F,t_n}(U_n) \big) , \label{strang}
\end{eqnarray}
where $k >0$ is the time stepsize and $\Psi_{k}^{F,t_n}(U)$ and
$\Psi_{k}^{F,t_n+\frac{k}{2}}(U)$ are the results of applying a
certain $p$th-order numerical method ($p\ge 1$) to the following
nonlinear differential problems:
\begin{eqnarray}
&U'(s)=F(t_n+s,U(s)), \qquad  & U'(s)=F(t_n+\frac{k}{2}+s,U(s)),
\nonumber
\end{eqnarray}
with initial condition $U(0)=U$ and $t_n=nk$ for $n \ge 0$.

\section{Time semidiscretization: exponential Lie-Trotter splitting}

In this section, we give the technique to generalize
Lie-Trotter exponential method, so that time order reduction is avoided
even with non-vanishing and time-dependent boundary conditions.
Besides, we prove the full-order of the local error of the time
semidiscretization.

\subsection{Description of the technique}

Whenever $M$ is a matrix, $e^{sM} V$ is the solution at $t=s$ of
\begin{eqnarray}
\label{odesystem}
\begin{array}{rcl}
\dot{U}(t)&=&M U(t),  \\
U(0)&=& V.
\end{array}
\end{eqnarray}
More generally, matrix $M$ can be substituted by the
infinitesimal generator $A_0$ of a $C_0$-semigroup in a
certain Banach space $X$. Then, the corresponding semigroup is
denoted by $e^{s A_0}$ and $e^{sA_0} v$, for $v\in D(A_0)\subset
X$, is the solution of the corresponding abstract differential
problem
\begin{eqnarray}
\dot{u}(t)&=&A_0 u(t), \nonumber \\
u(0)&=& v. \nonumber
\end{eqnarray}

When  $A_0$ is a linear (unbounded) operator associated to a
differential operator defined on $\Omega \subset \mathbb{R}^n$,
its domain $D(A_0)$ is formed by functions for which certain
boundary operator vanishes on the boundary of $\Omega$ (see
Example in Section 2).

Since we are interested in problems with nonvanishing boundary
conditions, as those in (\ref{laibvp}), we replace the exponential
matrices or semigroups with the solution of differential problems
where the boundary values must be specified in a clever way. More
precisely, we suggest to advance a stepsize from $u_n$ in the
following way. Firstly, we consider the solution of
\begin{eqnarray}
\label{vn}
 \left.\begin{array}{rcl}
v_{n}'(s)&=&A v_{n}(s),\\
v_{n}(0)&=&u_n,\\
\partial v_{n}(s)&=& \partial \hat{v}_{n}(s),
\end{array}\right.
\end{eqnarray}
where
\begin{eqnarray}
\hat{v}_{n}(s)=u(t_n)+s  A u(t_n). \label{vh}
\end{eqnarray}
Then, we consider the problem
\begin{eqnarray}
\label{wnk}
 \left.\begin{array}{rcl}
w_{n}'(s)&=&f(t_n+s,w_n(s)),\\
w_{n}(0)&=&v_n(k),
\end{array}\right.
\end{eqnarray}
and $u^{n+1}$ is obtained advancing a time step $k
\ge 0$ by means of a numerical integrator of order $p \ge 1$. That
is,
\begin{eqnarray}
u^{n+1}=\Psi_{k}^{f,t_n}(v_n(k)). \label{deflt}
\end{eqnarray}

Notice that we could have also started by integrating the
nonlinear part of the equation and then the linear and stiff one.
However, that would have led to a slightly more complicated
expression for the boundary in the linear part.

\begin{remark}
\label{r1} In order to calculate $\partial \hat{v}_n(s)$, apart
from $\partial u(t_n)=g(t_n)$, we also need $\partial A u(t_n)$,
for which we can use from (\ref{laibvp}),
$$
\partial A u(t_n)= \partial u'(t_n)-\partial f(t_n,u(t_n))=g'(t_n)-\partial f(t_n,u(t_n)).
$$
When the operator $\partial$ corresponds to a Dirichlet boundary
condition and the nonlinear term is given by
(\ref{nonlinearterm}),
$$\partial f(t_n,u(t_n))=\phi\circ g(t_n) + \partial h(t_n),$$
and $\partial \hat{v}_n(s)$ is exactly calculated from the given
data. However, when $\partial$ corresponds to a Robin or Neumann
boundary condition, $\partial A u(t_n)$ can only be calculated in
an approximated way. For that, we write the boundary condition as
\begin{eqnarray}
 \partial u= \alpha u |_{\partial \Omega}+\beta \partial_n u |_{\partial \Omega}=g, \quad \beta \neq 0,
 \label{nrbc}
\end{eqnarray}
with $\partial \Omega$ the boundary (or some part of it) of some
domain $\Omega$ and $\partial_n$ the normal derivative to that
boundary. Then,  when $f$ is again like in (\ref{nonlinearterm}),
it can be used that
\begin{eqnarray}
\partial f(t_n,u(t_n))= \alpha[\phi(u(t_n)|_{\partial \Omega})+h(t_n)|_{\partial \Omega}]
+\beta [ \phi'(u(t_n)|_{\partial \Omega})\partial_n
u(t_n)|_{\partial \Omega}+\partial_n h(t_n)|_{\partial \Omega}].
\nonumber
\end{eqnarray}
In this expression, $u(t_n)|_{\partial \Omega}$ can be substituted
by the numerical approximation at the previous step and
$\partial_n u(t_n)|_{\partial \Omega}$ by the result of applying
the following formula which comes from (\ref{nrbc})
$$ \partial_n u |_{\partial \Omega}=\frac{g(t_n)-\alpha u(t_n)|_{\partial \Omega}}{\beta}.$$
\end{remark}

\subsection{Local error of the time semidiscretization} In order
to study the local error, we consider the value which is obtained in
(\ref{deflt}) starting from $u(t_n)$ in (\ref{vn}). Then, we
obtain
\begin{eqnarray*}
\overline{u}_{n+1}=\Psi_{k}^{f,t_n}(\overline{v}_n(k)),
\end{eqnarray*}
where $\overline{v}_{n}(s)$ is the solution of
\begin{eqnarray}
\left.\begin{array}{rcl}
\overline{v}_{n}'(s)&=&A \overline{v}_{n}(s),\\
\overline{v}_{n}(0)&=&u(t_n),\\
\partial \overline{v}_{n}(s)&=&  \partial \hat{v}_n(s).
\end{array}\right.
\label{ovn}
\end{eqnarray}
with $\hat{v}_n(s)$ that in (\ref{vh}).

Before bounding the local error
$\rho_{n+1}=\bar{u}_{n+1}-u(t_{n+1})$, let us first study more
thoroughly  $\overline{v}_{n}(s)$.
\begin{lemma} The
solution of (\ref{ovn}) is given by
\begin{eqnarray}
\overline{v}_{n}(s) &=& u(t_n)+s A u(t_n)+  s^2 \varphi_2(s A_0)
A^2 u(t_n). \nonumber
\end{eqnarray}
where  $\varphi_{2}(z)$ is defined in (\ref{varphi}).
\label{teor1}
\end{lemma}
\begin{proof}
Notice that
\begin{eqnarray}
 \left.\begin{array}{rcl}
\overline{v}_{n}'(s)-\hat{v}_n'(s)&=&A \overline{v}_{n}(s)-A
u(t_n)=
A(\overline{v}_{n}(s)-\hat{v}_n(s))+A(\hat{v}_n(s)-u(t_n)) \nonumber \\
&=& A(\overline{v}_{n}(s)-\hat{v}_n(s))+ s A^2 u(t_n), \nonumber \\
\overline{v}_{n}(0)-\hat{v}_{n}(0)&=&0,\\
\partial(\overline{v}_{n}(s)-\hat{v}_{n}(s))&=&0.
\end{array}\right.
\end{eqnarray}
Then,
\begin{eqnarray}
\overline{v}_{n}(s)&=&\hat{v}_{n}(s)+ \int_0^s e^{(s-\tau)A_0}\tau
A^2 u(t_n) d \tau =u(t_n)+ s A u(t_n)+ s^2 \varphi_2(s A_0) A^2
u(t_n). \nonumber
\end{eqnarray}
\end{proof}

\begin{theorem}
\label{theolieerrorlocalsd}  Let us assume hypotheses
(A1)-(A6) and that the numerical integrator $\Psi_k$ integrates
(\ref{wnk}) with order $p\ge 1$ in $X$. Then, when integrating
(\ref{laibvp}) with Lie-Trotter method using the technique
(\ref{vn})-(\ref{deflt}), the local error satisfies
\begin{eqnarray*}
\rho_{n+1}\equiv \overline{u}_{n+1}-u(t_{n+1})=O(k^2).
\end{eqnarray*}
\end{theorem}
\begin{proof}
Denoting by $\overline{w}_n(s)$ the solution of
\begin{eqnarray}
\overline{w}_n'(s)&=&f(t_n+s, \overline{w}_n(s)) \nonumber \\
\overline{w}_n(0)&=&\overline{v}_n(k), \nonumber
\end{eqnarray}
it happens that
\begin{eqnarray}
\overline{w}_n(k)&=&\overline{v}_n(k)+k f(t_n,\overline{v}_n(k))+O(k^2) \nonumber \\
&=& u(t_n)+ k A u(t_n) +k f(t_n,u(t_n))+O(k^2)=u(t_{n+1})+O(k^2).
\nonumber
\end{eqnarray}
Then,
\begin{eqnarray}
\rho_{n+1}&=&\bar{u}_{n+1}-u(t_{n+1})=\Psi_{k}^{f,t_n}(\overline{v}_n(k))-u(t_{n+1}) \nonumber \\
&=&[\Psi_{k}^{f,t_n}(\overline{v}_n(k))-\overline{w}_n(k)]+[\overline{w}_n(k)-u(t_{n+1})]=O(k^{p+1})+O(k^2)=O(k^2).
 \nonumber
\end{eqnarray}

\end{proof}

\section{Time semidiscretization: exponential Strang splitting }
With the same idea as in Section 3, we describe now how to
generalize Strang exponential method in order to fight against
order reduction in time. Instead of achieving order $3$ for the
local error (as when integrating non-stiff ODEs), we will just
achieve order $2$ for it. This is due to the fact that we want to
guarantee that the boundary of the intermediate evolutionary
partial differential equation problem can be calculated in terms
of data. However, as we will see in Sections \ref{globalstrang}
and \ref{experimentos}, that will mean in practice no order
reduction for the global error because of a summation-by-parts
argument.

Notice that, instead of starting with the integration of the
linear part, as with Lie-Trotter method, we start with that of the
nonlinear and smooth one. This is because, in such a way, just one
stiff differential evolutionary problem per step arises for which
we must suggest a boundary.

\subsection{Description of the technique} For the time
integration of (\ref{laibvp}), we firstly consider the problem
\begin{eqnarray}
\label{strangsd-01}
\begin{array}{rcl}
v_{n}'(s)&=&f(t_n+s,v_n(s)),\\
v_{n}(0)&=&u_n,
\end{array}
\end{eqnarray}
and denote by $\Psi_{\frac{k}{2}}^{f,t_n}(u_n)$ the numerical
approximation of this problem after time $k/2$. Secondly, we
consider
\begin{eqnarray}
\label{strangsd-02}
\begin{array}{rcl}
w_{n}'(s)&=&A w_{n}(s),\\
w_{n}(0)&=&\Psi_{\frac{k}{2}}^{f,t_n}(u_n),\\
\partial w_{n}(s)&=&\partial \widehat{w}_{n}(s),\\
\end{array}
\end{eqnarray}
where
\begin{eqnarray}
\label{wnsombrero}
\widehat{w}_{n}(s)=u(t_n)+\frac{k}{2}f(t_n,u(t_n))+s A u(t_n),
\end{eqnarray}
which comes from approximating $v_n(\frac{k}{2})+ s A
v_n(\frac{k}{2})$. Thirdly, by considering
\begin{eqnarray}
\label{strangsd-03}
\begin{array}{rcl}
z_{n}'(s)&=&f(t_n+\frac{k}{2}+s,z_n(s)),\\
z_{n}(0)&=&w_n(k),
\end{array}
\end{eqnarray}
and advancing $k/2$ with the numerical integrator, we obtain
\begin{eqnarray}
\label{strangsd-final}
u_{n+1}=\Psi_{\frac{k}{2}}^{t_n+\frac{k}{2}}(w_n(k)).
\end{eqnarray}

\begin{remark}
\label{r2} Notice that the boundary values in (\ref{strangsd-02})
can be exactly or approximately calculated in terms of data
under the same considerations of Remark \ref{r1}.
\end{remark}
\subsection{Local error of the time semidiscretization}
\label{letss} In order to study the local error, we consider the
value $\overline{u}_{n+1}$ which is obtained in
(\ref{strangsd-final}) starting from $u_n=u(t_n)$ in
(\ref{strangsd-01}). Then, denoting by $\overline{w}_n(s)$ the
solution of (\ref{strangsd-02}) starting from
$\Psi_{\frac{k}{2}}^{f,t_n}(u(t_n))$, we have the following
result:
\begin{lemma}
\begin{eqnarray}
\overline{w}_n(s)&=&u(t_n)+\frac{k}{2} f(t_n, u(t_n))+s A
u(t_n)+e^{s A_0}
\big(\Psi_{\frac{k}{2}}^{f,t_n}(u(t_n))-u(t_n)-\frac{k}{2} f(t_n,u(t_n)) \big) \nonumber \\
&&+\frac{k}{2}s \varphi_1(s A_0) A f(t_n,u(t_n))+s^2 \varphi_2(s
A_0) A^2 u(t_n). \nonumber
\end{eqnarray}
\label{regstrang}
\end{lemma}
\begin{proof}
It can be noticed that
\begin{eqnarray}
\overline{w}_n'(s)-\hat{w}_n'(s)&=&A \overline{w}_n(s)-A u(t_n)= A(\overline{w}_n(s)-\hat{w}_n(s))+A \hat{w}_n(s)-A u(t_n) \nonumber \\
&=&A
(\overline{w}_n(s)-\hat{w}_n(s))+\frac{k}{2} A f(t_n,u(t_n))+ s A^2 u(t_n), \nonumber \\
\overline{w}_n(0)-\hat{w}_n(0)&=& \Psi_{\frac{k}{2}}^{f,t_n}(u(t_n))-u(t_n)-\frac{k}{2} f(t_n,u(t_n)), \nonumber \\
\partial(\overline{w}_n(s)-\hat{w}_n(s))&=&0.\nonumber
\end{eqnarray}
Then,
\begin{eqnarray}
\overline{w}_n(s)-\hat{w}_n(s)&=&e^{s A_0}(\overline{w}_n(0)-\hat{w}_n(0)) +\int_0^s e^{(s-\tau)A_0}[ \frac{k}{2} A f(t_n,u(t_n))+\tau A^2 u(t_n)] d\tau \nonumber \\
&=&e^{s A_0} \bigg( \Psi_{\frac{k}{2}}^{f,t_n}(u(t_n))-u(t_n)-\frac{k}{2} f(t_n,u(t_n))\bigg)\nonumber \\
&&+\frac{k}{2} s \varphi_1(s A_0) A f(t_n,u(t_n))+s^2 \varphi_2(s
A_0) A^2 u(t_n). \nonumber
\end{eqnarray}
\end{proof}
\begin{theorem}
Let us assume that hypotheses (A1)-(A6) are satisfied,
and that the numerical integrator $\Psi_k$ integrates
(\ref{strangsd-01}) and (\ref{strangsd-03}) with order $p\ge 1$ in
$X$. Then, when integrating (\ref{laibvp}) with Strang method
using the technique (\ref{strangsd-01})-(\ref{strangsd-03}), the
local error satisfies
\begin{eqnarray*}
\rho_{n+1}\equiv \overline{u}_{n+1}-u(t_{n+1})=O(k^2).
\end{eqnarray*}
\label{t4}
\end{theorem}
\begin{proof}
Notice that
\begin{eqnarray}
\rho_{n+1}&=&\bar{u}_{n+1}-u(t_{n+1})
=[\Psi_{\frac{k}{2}}^{f,t_n+\frac{k}{2}}(\overline{w}_n(k))-\overline{z}_n(\frac{k}{2})]+[\overline{z}_n(\frac{k}{2})-u(t_{n+1})] \nonumber \\
&=&\overline{w}_n(k)+ \frac{k}{2} f(t_n+\frac{k}{2}, \overline{w}_n(k))-u(t_n)-k Au(t_n)-k f(t_n,u(t_n)) \nonumber \\
&&+O(k^{p+1})+O(k^2) \nonumber \\
&=& O(k^2), \label{rho}
\end{eqnarray}
where the last equality is deduced from Lemma \ref{regstrang} and
the fact that $p\ge 1$.
\end{proof}
Moreover, assuming a bit more regularity of the functions $u$ and
$f$ and a bit more accuracy of the time numerical integrator for the nonlinear part,
we have the following result:
\begin{theorem}
\label{t5}Whenever, apart from hypotheses (A1)-(A6), $u\in
C^3([0,T], X)$, $f\in C^2([0,T]\times X, X)$ and
$f_u(\cdot,u(\cdot))f(\cdot,u(\cdot)), f_u(\cdot,u(\cdot)) A
u(\cdot) \in C([0,T], X)$, when integrating (\ref{laibvp}) with
Strang method using the technique
(\ref{strangsd-01})-(\ref{strangsd-final}) with a numerical
integrator $\Psi_k$ which is of order $p\ge 2$ for problems
(\ref{strangsd-01}) and (\ref{strangsd-03}), the local error
satisfies
\begin{eqnarray*}
A_0^{-1}\rho_{n+1}=O(k^3).
\end{eqnarray*}
\end{theorem}
\begin{proof}
Under the new assumptions, by explicitly writing the term in $k^2$
in (\ref{rho}), we have
\begin{eqnarray}
\rho_{n+1}&=&\frac{k^2}{4} f_t(t_n,u(t_n))+\frac{k^2}{4}f_u(t_n,u(t_n)) f(t_n,u(t_n))+\frac{k^2}{2}f_u(t_n,u(t_n)) A u(t_n)\nonumber \\
&&+\frac{k^2}{8}[f_t(t_n,u(t_n))+f_u(t_n,u(t_n))f(t_n,u(t_n))] \nonumber \\
&&+\frac{k^2}{8} e^{k A_0}[f_t(t_n,u(t_n))+f_u(t_n,u(t_n))f(t_n,u(t_n))] \nonumber \\
&&+\frac{k^2}{2}\varphi_1(k A_0) A f(t_n,u(t_n))+k^2 \varphi_2(k
A_0) A^2 u(t_n)-\frac{k^2}{2}u''(t_n) +O(k^3).
\end{eqnarray}
By applying now $A_0^{-1}$, considering (\ref{recurf}) and
simplifying terms and the notation for the sake of brevity,
\begin{eqnarray}
A_0^{-1}\rho_{n+1}&=&k^2 A_0^{-1}[\frac{3}{8}f_t+ \frac{3}{8}f_u f
+\frac{1}{2} f_u A u-\frac{1}{2} u''] +\frac{k^2}{8}( k
\varphi_1(k A_0)+A_0^{-1})(f_t+f_u  f)
\nonumber \\
&&
+\frac{k^2}{2}(k \varphi_2(k A_0)+A_0^{-1}) A f+ k^2(k \varphi_3(k A_0)+\frac{1}{2}A_0^{-1}) A^2 u +O(k^3) \nonumber \\
&=& k^2 A_0^{-1}[ \frac{1}{2}(f_t+f_u f)+\frac{1}{2} f_u Au
+\frac{1}{2} A f +\frac{1}{2} A^2 u-\frac{1}{2} u'']+O(k^3),
\nonumber
\end{eqnarray}
where, in order to see that the term in bracket vanishes, it
suffices to differentiate (\ref{laibvp}) once with respect to
time.

\end{proof}
\section{Spatial discretization}
\label{sd} Following the example in Section 2, we take
$X=C(\overline{\Omega})$ with the maximum norm and we consider a
certain grid $\Omega_h$ (of $\Omega$) over which the approximated
numerical solution will be defined. In this way, this numerical
approximation belongs to $C^N$, where $N$ is the number of nodes
in the grid, endowed with the the maximum norm $\|u_h\|_h=\|[u_1,
\ldots, u_N]^T\|_h= \max_{1 \le i \le N}|u_i|$.

Notice that, usually, when considering Dirichlet boundary
conditions, nodes on the boundary are not considered while, when
using Neumann or Robin boundary conditions, the nodes on the
boundary are taken into account.

In that sense, we consider the projection operator
\begin{eqnarray}
 P_h : X \to \mathbb{C}^N, \label{ph}
\end{eqnarray}
which takes a function to its values over the grid $\Omega_h$.
 On the other hand, the
operator $A$, when applied over functions which satisfy a certain
condition on the boundary $\partial u=g$, is discretized by means
of an operator
\begin{eqnarray*}
A_{h,g}: \mathbb{C}^N \to \mathbb{C}^N,
\end{eqnarray*}
which takes the boundary values into account. More precisely,
\begin{eqnarray*}
A_{h,g} U_h=A_{h,0}U_h+C_h g,
\end{eqnarray*}
where $A_{h,0}$ is the matrix which discretizes $A_0$ and $C_h: Y
\to \mathbb{C}^N$ is another operator, which is the one which
contains the information on the boundary.

We also assume that the source function $f$ has also sense as
function from $[0,T]\times \mathbb{C}^N$ on $\mathbb{C}^N$ and,
for each $t \in [0,T]$ and $u \in X$,
\begin{eqnarray}
\label{fandph} P_hf(t,u)= f(t, P_hu).
\end{eqnarray}
This fact is obvious when $f$ is given by (\ref{nonlinearterm}).
By using this, the following semidiscrete problem arises after
discretising (\ref{laibvp}) in space,
\begin{eqnarray}
\label{laibvpdisc} \left.
\begin{array}{rcl}
U'_h(t)&=&A_{h,0}U_h(t) + C_{h}g(t) +f(t,U_h(t)),\\
U_h(0)&=&P_h u(0),
\end{array}
\right.
\end{eqnarray}

The subsequent analysis is carried out under the following
hypotheses:
\begin{enumerate}
\item[(H1)] The matrix $A_{h,0}$ satisfies
\begin{enumerate}
\item $\|e^{tA_{h,0}}\|_h \le 1$,

\item $A_{h,0}$ is invertible and $\|A_{h,0}^{-1}\|_h \le C$ for
some constant $C$ which does not depend on $h$,
\end{enumerate}
where $\|\cdot\|_h$ is the norm operator obtained
from the maximum norm in $\mathbb{C}^N$.

\item[(H2)] We define the elliptic projection $R_{h}:D(A) \to
\mathbb{C}^N$  as the solution of
\begin{eqnarray}
A_{h,0} R_{h} u+ C_{h}\partial u=P_h Au. \label{rh}
\end{eqnarray}
We assume that there exists a subspace $Z \subset D(A)$, such
that, for $u \in Z$,
\begin{enumerate}
\item[(a)] $A_0^{-1} u \in Z$ and $e^{tA_0}u \in Z$, for $t \ge
0$.
\item[(b)] for some $\varepsilon_{h}$ and $\eta_{h}$ which are
both small with $h$,
\begin{equation}
\label{consistency} \hspace{-0.5cm} \left\| A_{h,0}({P_hu-R_{h}u})
\right\| \le \varepsilon_{h} \left\| u \right\|_Z, \quad
\left\|P_hu-R_{h}u \right\| \le \eta_{h} \left\| u \right\|_Z.
\end{equation}
(Although obviously, because of (H1), $\eta_h$ could be taken as
$C\varepsilon_h$, for some discretizations $\eta_h$ can decrease more quickly with $h$
than $\varepsilon_h$ and that leads to better error bounds in the
following section.)
\end{enumerate}
\item[(H3)] The nonlinear source $f$ belongs to
$C^1([0,T] \times X_h, X_h)$ and the derivative with respect to the variable $u_h$
is uniformly bounded in a neighbourhood of the solution where the numerical approximation stays.
\end{enumerate}

\begin{remark}
Hypothesis (H1a) can be deduced in our numerical experiments
by using the logarithmic norm of matrix
$A_{h,0}$, which is given by \cite{dahlquist}
\begin{eqnarray*}
\mu(A_{h,0})=\lim_{\tau \to 0^+}\frac{\| I+\tau
A_{h,0}\|-1}{\tau}.
\end{eqnarray*}
From the logarithmic norm, we obtain the bound
\begin{eqnarray*}
\| e^{tA_{h,0}}\| \le e^{t\mu(A_{h,0})}.
\end{eqnarray*}

In particular, with the maximum norm
($\|\mathbf{u}\|_\infty=\max_i|u_i|$), which is the one used in
our examples, we have
\begin{eqnarray*}
\mu_\infty(A)=\max_i\left( \Re a_{ii}+\sum_{j\not=
i}|a_{ji}|\right),
\end{eqnarray*}
which can be easily checked. For the matrices  $A_{h,0}$ in our numerical experiments, it is easily seen
that $\mu(A_{h,0})_\infty=0$ and (H1a) holds.
\end{remark}

\begin{remark}
From (H3), the non linear term $f$ satisfies, for some Lipschitz
constant $L$ independent of $t\in [0,T]$ and the maximum norm,
\begin{eqnarray}
\label{locallipschtzdisc}
  \|f(t, v_h)-f(t,u_h)\|_h\le  L\|v_h-u_h \|_h,
\end{eqnarray}
when $u_h,v_h$ belong to a compact set. In particular,
we will be interested in considering as this set a neighborhood of the
exact solution where the numerical approximation stays.
\end{remark}

\section{Full discretization: exponential Lie-Trotter splitting}

\subsection{Final formula for the
implementation} We apply the above space discretization
to the evolutionary problems (\ref{vn}) and (\ref{wnk}) and we
obtain $V_{h,n}(s), W_{h,n}(s)$ in $\mathbb{C}^N$ as the solutions
of
\begin{eqnarray}
V'_{h,n}(s)&=&A_{h,0}V_{h,n}(s)+C_h \partial \hat{v}_{n}(s), \nonumber\\
V_{h,n}(0)&=&U_{h,n}, \label{Vhn}
\end{eqnarray}
where $\hat{v}_{n}(s)$ is that in (\ref{vh}), $U_{h,n}\in
\mathbb{C}^N$ is the numerical solution in the interior of the
domain after full discretization at $n$ steps, and
\begin{eqnarray}
W'_{h,n}(s)&=& f(t_n+s, W_{h,n}(s)), \label{Whnk}\\
W_{h,n}(0)&=&V_{h,n}(k). \nonumber
\end{eqnarray}
By using the variations of constants formula and the definition of
the functions $\varphi_1$ and $\varphi_2$ in (\ref{varphi}),
\begin{eqnarray}
V_{h,n}(k)&=&e^{k A_{h,0}} U_{h,n}+\int_0^k e^{(k-s) A_{h,0}}\big[C_h\partial [u(t_n)+s A u(t_n)] \big]ds \nonumber \\
       &=&e^{k A_{h,0}} U_{h,n}+ k \varphi_1( k A_{h,0})C_h  g(t_n) \nonumber \\
&&+ k^2 \varphi_2( k A_{h,0}) C_h(g'(t_n)-\partial f(t_n,u(t_n))),
\label{ffv}
\end{eqnarray}
and the numerical solution at step $n+1$ is therefore given by
\begin{eqnarray}
U_{h,n+1}=\Psi_{k}^{f,t_n}(V_{h,n}(k)), \label{Uhn_1}
\end{eqnarray}
where $\Psi_{k}^{f,t_n}$ stands for the previously mentioned
numerical integrator applied to (\ref{Whnk}).

Moreover, we will take, as initial condition,
\begin{eqnarray}
U_{h,0}=P_h u(0). \label{initn}
\end{eqnarray}
\begin{remark}
\label{nor} Notice that, when
$$\partial u(t_n)=\partial A u(t_n)=0,$$
it is also deduced from (\ref{laibvp}) that $\partial f(t_n,
u(t_n))=0$. Therefore, formulas (\ref{ffv})-(\ref{Uhn_1}) just
reduce to the standard time integration with Lie-Trotter method of
the differential system which arises after discretizing
(\ref{laibvp}) directly in space  (see (\ref{laibvpdisc})):
$$U_h'(t)= A_{h,0} U_h(t)+f(t, U_h(t)).$$
Because of that, with the results which follow, we will be implicitly proving that there is no
order reduction in the local error with the standard Lie-Trotter
method under these assumptions.
\end{remark}

\begin{remark}
We notice that,  when $k$ is fixed, $e^{k A_{h,0}}$ and
$\varphi_{j}(k A_{h,0})$ could be calculated once and for all at
the very beginning.  Besides, as better explained in the numerical
experiments, $C_h$ will be represented by a matrix of dimension
$O(\hat{N}^d)\times O(\hat{N}^{d-1})$ where $d$ is the dimension
of the problem and $\hat{N}$ the number of grid points in each
direction. Then, $\varphi_{j}(k A_{h,0})C_h$ ($j=1,2$) will be
represented by  matrices of the same order and therefore the
computational cost of calculating the product of those matrices
times the information on the boundary values is
$O(\hat{N}^{2d-1})$, which is negligible compared with
$O(\hat{N}^{2d})$, which corresponds to the calculation of the
product of  $e^{k A_{h,0}}$ times a vector of size $O(\hat{N}^d)$.
On the other hand, for fixed and variable timestepsize $k$, Krylov
techniques can also be used to calculate the terms in (\ref{ffv})
without explicitly calculating $e^{k A_{h,0}}$, $\varphi_{j}(k
A_{h,0})$ ($j=1,2$). As suggested in \cite{Grimm}, it seems in
principle cheaper to calculate the terms containing the
$\varphi_j$-functions than those corresponding to the
exponentials.
\end{remark}

\subsection{Local errors}
In order to define the local error, we consider

\begin{eqnarray}
\overline{U}_{h}^{n+1}=\Psi_{k}^{f, t_n}(\overline{V}_{h,n}(k)),
\label{Ub}
\end{eqnarray}
where $\overline{V}_{h,n}(s)$ is the solution of
\begin{eqnarray}
\label{oVhnk}
\begin{array}{rcl}
\overline{V}'_{h,n}(s)&=&A_{h,0}\overline{V}_{h,n}(s)+C_{h}\partial \hat{v}_{n}(s),\\
\overline{V}_{h,n}(0)&=&P_h u(t_n).
\end{array}
\end{eqnarray}
with $\hat{v}_n(s)$ that in (\ref{vh}). We now define the local
error at $t=t_n$ as
\begin{eqnarray*}
\rho_{h,n}=P_{h} u(t_n)-\overline{U}_{h,n},
\end{eqnarray*}
and study its  behaviour in the following theorem.
\begin{theorem}
\label{theolocalerrorfull} Let us assume hypotheses
(A1)-(A6), that $\Psi_k$ integrates (\ref{Whnk}) with order $p\ge
1$ and (H1)-(H3). Then, when integrating (\ref{laibvp}) with
Lie-Trotter method as described in (\ref{ffv})-(\ref{Uhn_1}),
whenever $u$ satisfies
\begin{eqnarray}
u, A u, A^2 u \in C([0,T],  Z),
\label{regs}
\end{eqnarray}
for the space $Z$  in (H3), the local error after full
discretization satisfies
\begin{eqnarray}
\rho_{h, n+1}=O(k\varepsilon_{h}+ k^{2}), \quad A_{h,0}^{-1}
\rho_{h,n+1}=O(k \eta_h+k^2). \label{lefdlt}
\end{eqnarray}
where $\varepsilon_{h}$ and $\eta_h$ are those in
(\ref{consistency}).
\end{theorem}
\begin{proof} Notice that
\begin{eqnarray}
\overline{U}_{h}^{n+1}&=&\Psi_{k}^{f,t_n}(\overline{V}_{h,n}(k))=\overline{V}_{h,n}(k)+k f(t_n,
\overline{V}_{h,n}(k))+O(k^2). \nonumber
\end{eqnarray}

On the other hand, making the difference between (\ref{oVhnk}) and
(\ref{ovn}) multiplied by $P_h$,
\begin{eqnarray}
\overline{V}_{h,n}'(s)-P_h \overline{v}_n'(s)&=& A_{h,0} (\overline{V}_{h,n}(s)-
P_h \overline{v}_n(s))+A_{h,0}(P_h-R_h)\overline{v}_n(s), \nonumber \\
\overline{V}_{h,n}(0)-P_h \overline{v}_n(0)&=& 0. \nonumber
\end{eqnarray}
Then,
\begin{eqnarray}
\overline{V}_{h,n}(k)&=&P_h \overline{v}_n(k)+\int_0^k e^{(k-s)
A_{h,0}}A_{h,0}(P_h-R_h) \overline{v}_n(s)ds
\nonumber \\
&=& P_h u(t_n)+ k P_h A u(t_n)+O(k^2)+ O(k\varepsilon_h),
\label{ovhn}
\end{eqnarray}
where the last equality comes from  Lemma \ref{teor1}, (H1) and
(H2). From the definition of $\rho_{h,n}$,
\begin{eqnarray*}
\rho_{h,n+1}&=&P_{h}u(t_{n+1})-\overline{U}_{h,n+1}= P_{h}(u(t_{n+1})-
\overline{u}_{n+1})+( P_{h}\overline{u}_{n+1}-\overline{U}_{h,n+1}) \nonumber \\
&=& P_h \rho_{n+1}+
P_{h}\overline{u}_{n+1}-\overline{V}_{h,n}(k)-k f(t_n,
\overline{V}_{h,n}(k))+O(k^2).
\end{eqnarray*}
Considering now Theorem \ref{theolieerrorlocalsd}, (\ref{ph}),
(\ref{ovhn}), (H3) and the fact that, because of Lemma
\ref{teor1},
\begin{eqnarray}
\overline{u}_{n+1}&=&\Psi_{k}^{f,t_n}(\overline{v}_n(k))=\overline{v}_n(k)+ k f(t_n,\overline{v}_n(k))+O(k^2) \nonumber \\
&=& u(t_n)+ k A u(t_n)+ k f(t_n,u(t_n))+O(k^2), \nonumber
\end{eqnarray}
the first part of the theorem is proved. To prove the second bound
in (\ref{lefdlt}), it suffices to apply the uniformly bounded
matrix $A_{h,0}^{-1}$ to the above formulas and to take the second
part of (\ref{consistency}) into account.
\end{proof}

\subsection{Global errors}
We now study the global errors at $t=t_n$, which are given by
\begin{eqnarray*}
e_{h,n}=P_h u(t_n)-U_{h,n}.
\end{eqnarray*}
\begin{theorem}
\label{theoglobalerrorfull} Under the same assumptions of Theorem
\ref{theolocalerrorfull}, and assuming also that  $\partial
f(t_n,u(t_n))$ can be calculated exactly from data according to
Remark \ref{r1}, the global error which turns up when integrating
(\ref{laibvp}) through formulas (\ref{ffv})-(\ref{initn})
satisfies
\begin{eqnarray*}
e_{h,n}=O(k+ \varepsilon_{h}),
\end{eqnarray*}
where  $\varepsilon_{h}$ is that in (\ref{consistency}).
\end{theorem}
\begin{proof}
It suffices to notice that
\begin{eqnarray}
e_{h,n+1} &=& [P_h u(t_{n+1})-\overline{U}_{h}^{n+1}]+[\overline{U}_{h}^{n+1}-U_{h,n+1}] \nonumber \\
&=&\rho_{h,n+1}+\Psi_{k}^{f,t_n}(\overline{V}_{h,n}(k))-\Psi_{k}^{f,t_n}(V_{h,n}(k)) \nonumber \\
&=&\rho_{h,n+1}+\overline{W}_{h,n}(k)-W_{h,n}(k)+O(k^{p+1}) ,
\nonumber
\end{eqnarray}
where $\overline{W}_{h,n}(k)$ is the solution of (\ref{Whnk})
with initial condition $\overline{V}_{h,n}(k)$, and the definition
of $\rho_{h,n+1}$, (\ref{Uhn_1}) and (\ref{Ub}) have been used.
Then, considering (\ref{Whnk}),
\begin{eqnarray}
\overline{W}'_{h,n}(t)-W'_{h,n}(s)&=& f(t_n+s, \overline{W}_{h,n}(s))-f(t_n+s, W_{h,n}(s)), \nonumber\\
\overline{W}_{h,n}(0)-W_{h,n}(0)&=&\overline{V}_{h,n}(k)-V_{h,n}(k),
\nonumber
\end{eqnarray}
and
\begin{eqnarray}
\overline{W}_{h,n}(t)-W_{h,n}(t)&=&\overline{W}_{h,n}(0)-W_{h,n}(0)+
\int_0^t [ f(t_n+s,\overline{W}_{h,n}(s))-f(t_n+s,W_{h,n}(s))] ds
\nonumber\\
&=&\overline{V}_{h,n}(k)-V_{h,n}(k)+ \int_0^t [
f(t_n+s,\overline{W}_{h,n}(s))-f(t_n+s,W_{h,n}(s))] ds. \nonumber
\end{eqnarray}
Taking norms,
\begin{eqnarray}
\|\overline{W}_{h,n}(t)-W_{h,n}(t)\|&\le&
\|\overline{V}_{h,n}(k)-V_{h,n}(k)\|+ \int_0^t \|
f(t_n+s,\overline{W}_{h,n}(s))-f(t_n+s,W_{h,n}(s))\| ds
\nonumber\\
&\le& \|\overline{V}_{h,n}(k)-V_{h,n}(k)\|+ \int_0^k L\|
\overline{W}_{h,n}(s)-W_{h,n}(s)\| ds,
\end{eqnarray}
where (H3) has been used. We can
then apply Gronwall lemma and deduce that
\begin{eqnarray*}
\|\overline{W}_{h,n}(t)-W_{h,n}(t)\| &\le&
e^{Lt}\|\overline{V}_{h,n}(k)-V_{h,n}(k)\|.
\end{eqnarray*}

Moreover,
\begin{eqnarray}
\overline{W}_{h,n}(k)-W_{h,n}(k)&=&\overline{V}_{h,n}(k)-V_{h,n}(k)+
\int_0^k [ f(t_n+s,\overline{W}_{h,n}(s))-f(t_n+s,W_{h,n}(s))]
ds\nonumber\\
&=&\overline{V}_{h,n}(k)-V_{h,n}(k)+E(\overline{V}_{h,n}(k),V_{h,n}(k)),
\end{eqnarray}
where
\begin{eqnarray*}
\|E(\overline{V}_{h,n}(k),V_{h,n}(k))\|&=&\|\int_0^k [
f(t_n+s,\overline{W}_{h,n}(s))-f(t_n+s,W_{h,n}(s))]\|\\
&\le&L\int_0^k e^{sL}\|\overline{V}_{h,n}(k)-V_{h,n}(k)\| ds \le
kC\|\overline{V}_{h,n}(k)-V_{h,n}(k)\|,
\end{eqnarray*}
for a constant $C$ which is independent of $k$, when $k$ is small
enough.

On the other hand, by making the difference between (\ref{Vhn})
and (\ref{oVhnk}),
\begin{eqnarray}
\overline{V}_{h,n}'(s)-V_{h,n}'(s)&=& A_{h,0} (\overline{V}_{h,n}(s)-V_{h,n}(s)), \label{Vhn1} \\
\overline{V}_{h,n}(0)-V_{h,n}(0)&=&P_h u(t_n)- U_{h}^{n},
\nonumber
\end{eqnarray}
from what
$$
\overline{V}_{h,n}(k)-V_{h,n}(k)=e^{k A_{h,0}} (P_h
u(t_n)-U_{h}^{n}),$$ and then
$$
e_{h,n+1}=e^{k A_{h,0}} e_{h,n}+ \rho_{h,n+1}+k \bar{E}_h
(U_{h}^{n}, P_h u(t_n)),$$ where, for some constant $\bar{C}$,
\begin{eqnarray}
\|\bar{E}_h (U_{h}^{n}, P_h u(t_n))\|_h\le \bar{C} \|e_{h,n}\|_h.
\label{Eh}
\end{eqnarray}
From here,
\begin{eqnarray}
e_{h,n}&=& e^{n k A_{h,0}} e_{h,0}+\sum_{l=1}^n e^{(n-l)k A_{h,0}}
\rho_{h,l}+k \sum_{l=0}^{n-1} e^{(n-l-1)k A_{h,0}}
\bar{E}_h(U_{h}^{l}, P_h u(t_l)). \label{recurlt}
\end{eqnarray}
As $e_{h,0}= P_h u(0)- U_{h,0}=0$, by taking norms,
$$
\|e_{h,n}\|_h \le O(k+\varepsilon_h)+ k \bar{C} \sum_{l=0}^{n-1}
\|e_{h,l}\|_h,
$$
and using the discrete Gronwall lemma, the result follows.
\end{proof}

Another finer result is the following, which will be very useful
for non-Dirichlet boundary conditions.  When $\partial
\hat{v}_n(s)$ is numerically approximated through $U_{h,n}$,
according to the comments made in Remark \ref{r1}, we will denote
by $C_{h,n}^*(U_{h,n})$ to the vector which approximates $C_h
\partial f(t_n, u(t_n))$.

\begin{theorem}
\label{theoglobalerrorfull2} Let us assume the same
hypotheses of Theorem \ref{theolocalerrorfull}, that $u$ belongs to
$C^3([0,T],X)$ and $ Af(\cdot, u(\cdot)),f_t(\cdot, u(\cdot)),
f_u(\cdot, u(\cdot))$ to  $C^1([0,T],X)$, and also that there exist
constants $C$ and $C'$, independent of $h$, such that
\begin{eqnarray}
\|A_{h,0}^{-1} [C_{h,n}^{*}(U_{h}^n)-C_h \partial f(t_n, u(t_n))]\| &\le& C \|U_{h}^n -P_h u(t_n)\|, \label{cnum} \\
\|k A_{h,0} \sum_{r=1}^{n-1} e^{r k A_{h,0}}\| &\le& C' , \quad 0
\le n k \le T. \label{spp}
\end{eqnarray}
Then, the global error satisfies
\begin{eqnarray*}
e_{h,n}=O(k+ \eta_h+k\varepsilon_{h}),
\end{eqnarray*}
where  $\eta_h$ and $\varepsilon_{h}$ are those in
(\ref{consistency}).
\end{theorem}

\begin{proof}
The proof is very similar to that of Theorem
\ref{theoglobalerrorfull} with the difference that now
(\ref{Vhn1}) must be substituted by
\begin{eqnarray}
\overline{V}_{h,n}'(s)-V_{h,n}'(s)&=& A_{h,0}
(\overline{V}_{h,n}(s)-V_{h,n}(s))-s [C_h \partial f(t_n,
u(t_n))-C_{h,n}^{*}(U_{h,n})]. \nonumber
\end{eqnarray}
Therefore,
\begin{eqnarray}
\lefteqn{\overline{V}_{h,n}(k)-V_{h,n}(k)} \nonumber \\
&=& e^{k A_{h,0}}(P_h u(t_n)-U_n^n)-k^2 \varphi_2(k A_{h,0})[ C_h
\partial f(t_n,
u(t_n))-C_{h,n}^{*}(U_{h,n})] \nonumber \\
&=&e^{k A_{h,0}}(P_h u(t_n)-U_n^n)-k [\varphi_1(k
A_{h,0})-I]A_{h,0}^{-1} [C_h \partial f(t_n,
u(t_n))-C_{h,n}^{*}(U_{h,n})] \nonumber \\
&=& e^{k A_{h,0}}(P_h u(t_n)-U_n^n)+ O(k \|e_{h,n}\|), \nonumber
\end{eqnarray}
where the definition of $\varphi_1$ in (\ref{varphi}) has been
considered as well as (\ref{recurf}) and (\ref{cnum}). From this,
(\ref{recurlt}) still applies for some other function $\bar{E}_h$
also satisfying (\ref{Eh}). Then, we write one of the terms in
(\ref{recurlt})  as
\begin{eqnarray}
\lefteqn{\sum_{l=1}^n e^{k (n-l) A_{h,0}} \rho_{h,l}} \nonumber \\
&& =\big( \sum_{r=1}^{n-1}
e^{r k A_{h,0}}\big) \rho_{h,1}+\sum_{j=2}^{n-1} \big(
\sum_{r=1}^{j-1} e^{r k
A_{h,0}}\big)(\rho_{h,n-j+1}-\rho_{h,n-j})+\rho_{h,n}.
\label{decomp}
\end{eqnarray}
As the first term in this decomposition can be written as
\begin{eqnarray}
\big( \sum_{r=1}^{n-1} e^{r k A_{h,0}}\big) \rho_{h,1}=\big( k
A_{h,0} \sum_{r=1}^{n-1} e^{r k A_{h,0}}\big) \frac{1}{k}
A_{h,0}^{-1}\rho_{h,1}, \label{ftd}
\end{eqnarray}
applying (\ref{spp}) and Theorem \ref{theolocalerrorfull}, this
term is proved to be $O(k+\eta_h)$. As for the second term in
(\ref{decomp}), because of (A4) and (A5), the term in $k^2$ in
$A_{0}^{-1} \rho_{n+1}$ is differentiable with respect to time
$t_n$ and therefore, $A_0^{-1}(\rho_{n-j+1}-\rho_{n-j})=O(k^3)$.
When this is used in the local error for the full discretization,
$A_{h,0}^{-1}(\rho_{h,n-j+1}-\rho_{h,n-j})=O(k^3+k^2 \eta_h)$,
from what
\begin{eqnarray}
\sum_{j=2}^{n-1} \big( \sum_{r=1}^{j-1} e^{ r k A_{h,0}}\big)
(\rho_{h,n-j+1}-\rho_{h,n-j})&=&\sum_{j=2}^{n-1}
\big( k A_{h,0}\sum_{r=1}^{j-1} e^{r k A_{h,0}}\big)\frac{1}{k} A_{h,0}^{-1}(\rho_{h,n-j+1}-\rho_{h,n-j}) \nonumber \\
&=&O(k +\eta_h). \nonumber
\end{eqnarray}
Finally, using Theorem \ref{theoglobalerrorfull} for the last term
in (\ref{decomp}), it is clear that
$$\sum_{l=1}^n e^{k (n-l) A_{h,0}} \rho_{h,l}=O(k+\eta_h+k\varepsilon_h),$$
and the proof of the theorem follows by applying discrete Gronwall lemma to
$$\|e_{h,n}\|_h \le O(k+\eta_h+k\varepsilon_h)+k \bar{\bar{C}}\sum_{l=0}^{n-1} \|e_{h,l}\|_h .$$
\end{proof}
\begin{remark}
Bound (\ref{spp}) has been proved in \cite{HO} for analytic
semigroups, covering the case in which the linear operator in
(\ref{laibvp}) corresponds to that of a parabolic problem.
\end{remark}

\section{Full discretization: exponential Strang splitting}

\subsection{Final formula for the
implementation} Firstly, we apply the space discretization in
Section \ref{sd} to the evolutionary problem (\ref{strangsd-01})
and we obtain $V_{h,n}(s)\in \mathbb{C}^N$  as the solution of
\begin{eqnarray}
V'_{h,n}(s)&=&f(t_n+s, V_{h,n}(s)), \nonumber\\
V_{h,n}(0)&=&U_{h,n}. \label{Vhns}
\end{eqnarray}

 We will have to use the numerical integrator
$\Psi$ in order to approximate the solution of this problem.
Then, we define
\begin{eqnarray}
V_h^n=\Psi_{\frac{k}{2}}^{f, t_n} (U_h^n). \label{Vhnsf}
\end{eqnarray}

As a second step, discretizing (\ref{strangsd-02}), we consider
$W_{h,n}(s)\in \mathbb{C}^N$ as the solution of
\begin{eqnarray}
W'_{h,n}(s)&=& A_{h,0} W_{h,n}(s)+C_h \partial \hat{w}_n(s), \label{Whns}\\
W_{h,n}(0)&=&V_{h}^n. \nonumber
\end{eqnarray}
where $\hat{w}_n(s)$ is that in (\ref{wnsombrero}). By using the
variations of constants formula and the definition of the
functions $\varphi_1$ and $\varphi_2$ in (\ref{varphi}), we can
solve this problem exactly and we get
\begin{eqnarray}
W_{h,n}(k)&=&e^{k A_{h,0}} V_{h}^n+\int_0^k e^{(k-s) A_{h,0}}C_{h}  \partial \hat{w}_n(s) ds \nonumber \\
       &=&e^{k A_{h,0}} V_{h}^n+ k \varphi_1( k A_{h,0})C_h [g(t_n)+\frac{k}{2} \partial
       f(t_n,u(t_n))]
\nonumber \\&& + k^2 \varphi_2( k A_{h,0}) C_h [g'(t_n)-\partial
f(t_n,u(t_n)] \label{ffws}
\end{eqnarray}
Finally, from (\ref{strangsd-03}), we consider $Z_{h,n}(s)\in
\mathbb{C}^N$ as the solution of
\begin{eqnarray}
Z'_{h,n}(s)&=&f(t_n+\frac{k}{2}+s, Z_{h,n}(s)), \label{Zhns}\\
Z_{h,n}(0)&=&W_{h,n}(k), \nonumber
\end{eqnarray}
and, numerically integrating this problem, we
obtain
\begin{eqnarray}
U_h^{n+1}=\Psi_{\frac{k}{2}}^{f, t_n+\frac{k}{2}}
(W_{h,n}(k)). \label{Uhn_1s}
\end{eqnarray}
\begin{remark}
Similar comments to those in Remark \ref{nor} apply here.
Therefore, when $\partial u(t)=\partial A u(t)=0$, with the results which follow we will be implicitly proving that the standard
discretization with Strang method gives to rise to order $2$ for
the local error.
\end{remark}
\subsection{Local error}
In order to define the local error, we consider
$$
\overline{U}_{h,n+1}=\Psi_{\frac{k}{2}}^{f,
t_n+\frac{k}{2}}(\overline{W}_{h,n}(k)),
$$
where $\overline{W}_{h,n}(s)$ is the solution of
\begin{eqnarray}
\overline{W}_{h,n}'(s)&=& A_{h,0} \overline{W}_{h,n}(s)+ C_h \partial \hat{w}_n(s), \label{owhns} \\
\overline{W}_{h,n}(0)&=& \Psi_{\frac{k}{2}}^{P_h f, t_n}(P_h
u(t_n)), \nonumber
\end{eqnarray}
and $\hat{w}_n(s)$ is that in (\ref{wnsombrero}).

Then, for the local error $\rho_{h,n+1}=P_h u(t_{n+1})-
\overline{U}_{h,n+1}$, we have the following results:
\begin{theorem}
\label{theolocalerrorfulls}  Let us assume the same hypotheses of Theorem \ref{t5},
 also that $f\in C^2([0,T]\times X_h,X_h)$, that  $\Psi_k$ integrates
(\ref{Zhns}) with order $p\ge 2$ and (H1)-(H3). Then, when integrating
(\ref{laibvp}) with Strang method as described in (\ref{Vhnsf}),
(\ref{ffws}) and (\ref{Uhn_1s}), whenever $u$ and $f$ satisfy
\begin{eqnarray}
u(t_n), A u(t_n), A^2 u(t_n), f(t_n, u(t_n)), A f(t_n, u(t_n))
 \in Z, \label{regs2}
\end{eqnarray}
for the space $Z$  in (H2) and $\Psi_{\frac{k}{2}}^{f,t_n}$ leaves
this space invariant, the local error after full discretization
satisfies
\begin{eqnarray*}
\rho_{h, n+1}=O(k\varepsilon_{h}+ k^{2}), \quad
A_{h,0}^{-1}\rho_{h, n+1}=O(k\eta_{h}+k^2
\varepsilon_h+ k^{3}),
\end{eqnarray*}
where $\varepsilon_{h}$ and $\eta_h$ are those in
(\ref{consistency}).
\end{theorem}
\begin{proof}
Notice that
\begin{eqnarray}
\rho_{h,n+1}
&=&P_h u(t_n)-\overline{U}_{h,n+1}=P_h \rho_{n+1}+P_h
\overline{u}_{n+1}-\overline{U}_{h,n+1}
\nonumber \\
&=&P_h \rho_{n+1}+
P_h \Psi_{\frac{k}{2}}^{f,t_n+\frac{k}{2}}(\overline{w}_n(k))-
\Psi_{\frac{k}{2}}^{f,t_n+\frac{k}{2}}(\overline{W}_{h,n}(k)) \nonumber \\
&=& P_h \rho_{n+1}+P_h\overline{w}_n(k)-\overline{W}_{h,n}(k)+\frac{k}{2}f(t_n+\frac{k}{2}, P_h \overline{w}_n(k))
-\frac{k}{2}f(t_n+\frac{k}{2},\overline{W}_{h,n}(k)) \nonumber \\
&&+\frac{k^2}{4}[(f_t+f_u f)(t_n+\frac{k}{2},P_h \overline{w}_n(k))-(f_t+f_u f)(t_n+\frac{k}{2},\overline{W}_{h,n}(k))]+O(k^3), \label{rhohstrang}
\end{eqnarray}
where we have used (\ref{fandph}), the fact that $\Psi_{\frac{k}{2}}^{f,t_n}$ integrates (\ref{Zhns}) with order $p\ge 2$ and that $f\in C^2([0,T]\times X_h, X_h)$.
Now, from (\ref{owhns}) and the definition of $\overline{w}_n(s)$ in
Subsection \ref{letss},
\begin{eqnarray}
P_h \overline{w}_n'(s)-\overline{W}_{h,n}'(s)&=&P_h A \overline{w}_n(s)-A_{h,0} \overline{W}_{h,n}(s)-
C_h \partial \hat{w}_n(s) \nonumber \\
&=&A_{h,0}(P_h \overline{w}_n(s)-\overline{W}_{h,n}(s))+
A_{h,0}(R_h - P_h)\overline{w}_n(s), \nonumber
\end{eqnarray}
where the last term is $O(\varepsilon_h)$ according to Lemma
\ref{regstrang}, (\ref{regs2}) and  (H2). Moreover,
 again by (\ref{fandph}),
$$
P_h \overline{w}_n(0)-\overline{W}_{h,n}(0)= P_h
\Psi_{\frac{k}{2}}^{f,t_n}(u(t_n))-\Psi_{\frac{k}{2}}^{f,t_n}(P_h
u(t_n))=0.
$$
Then,
$$
P_h \overline{w}_n(k)-\overline{W}_{h,n}(k)= \int_0^k e^{(k-s)
A_{h,0}}A_{h,0}(R_h - P_h)\overline{w}_n(s)ds=O(k \varepsilon_h),
$$
and the first result follows by using Theorem \ref{t4}.

For the second result, applying $A_{h,0}^{-1}$ to
(\ref{rhohstrang})and using the second formula in
(H2b), it follows in the same way that
$$A_{h,0}^{-1} \rho_{h,n+1}= A_{h,0}^{-1} P_h \rho_{n+1}+ O(k \eta_h)+O(k^2 \varepsilon_h)+O(k^3).$$
Now, the key of the proof is that, if
$\omega_{n+1}=A_0^{-1} \rho_{n+1}$, $R_h \omega_{n+1}=A_{h,0}^{-1}
P_h \rho_{n+1}$. This comes from the fact that $\omega_{n+1}$ is
the solution of
$$ A \omega_{n+1}= \rho_{n+1}, \qquad \partial \omega_{n+1}=0,$$
from what
$
A_{h,0} R_h \omega_{n+1}= P_h \rho_{n+1}.$ Moreover, we will take
into account that
$$
R_h \omega_{n+1}= P_h \omega_{n+1}+(R_h-P_h)\omega_{n+1}=
O(k^3)+O(k \eta_h),
$$
where we have used Theorem \ref{t5} for the bound of the first
term. As for the second, we have used the definition of
$\rho_{n+1}$,  (\ref{regs2}),  Lemma \ref{regstrang}, hypothesis
(H2a) and the fact that $\Psi_{\frac{k}{2}}^{f,t_n}$ leaves $Z$
invariant. Because of this, $\omega_{n+1}\in Z$ and
$\|\omega_{n+1}\|_Z=O(k)$, and using  (H2b)  the result follows.
\end{proof}
\subsection{Global errors}
\label{globalstrang}
From the first result for the local error $\rho_{h,n+1}$ in Theorem \ref{theolocalerrorfulls}, a classical argument for the global error gives $e_{h,n}=O(k+\varepsilon_h)$. However, we also have this finer result, which will be
very useful for Dirichlet and non-Dirichlet boundary conditions
and parabolic problems:
\begin{theorem}
\label{theoglobalerrorfulls} Let us assume the same
hypotheses of Theorem \ref{theolocalerrorfulls}, that the bound
(\ref{spp}) is satisfied, and
\begin{enumerate}{}
\item[(i)]
 $u\in C^4([0,T], X)$, $f \in
C^3([0,T]\times X, X)$,$u(t)\in D(A^3)$ and
$f(t,u(t))\in D(A^2)$ for all $t\in [0,T]$ and $A^3 u, A^2
f(\cdot,u(\cdot)) \in C^1([0,T],X)$,\item[(ii)] $\partial
f(t_n,u(t_n))$ is calculated exactly or just approximated from data
according to Remark \ref{r2} and  (\ref{cnum}), \item[(iii)]
the term in $k^3$ for the local error when
integrating (\ref{strangsd-03}) with  $\Psi_k$ is differentiable with respect to
$t_n$,
\item[(iv)]$ u(\cdot), A u(\cdot), A^2 u(\cdot), f(\cdot,
u(\cdot)), A f(\cdot, u(\cdot))  \in C^1([0,T],Z)$.
\end{enumerate}
Then, the global error which turns up when
integrating (\ref{laibvp}) through (\ref{Vhnsf}), (\ref{ffws}) and
(\ref{Uhn_1s}), satisfies
\begin{eqnarray*}
e_{h,n}=O(k^2+ k \varepsilon_{h}+\eta_h),
\end{eqnarray*}
where  $\varepsilon_{h}$ and $\eta_h$ are those in
(\ref{consistency}).
\end{theorem}
\begin{proof}
Notice that
\begin{eqnarray}
e_{h,n+1}&=& [P_h u(t_{n+1})-\overline{U}_{h,n+1}]+[\overline{U}_{h,n+1}-U_{h,n+1}] \nonumber \\
&=&\rho_{h,n+1}+\Psi_{\frac{k}{2}}^{f,t_n+\frac{k}{2}}(\overline{W}_{h,n}(k))
-\Psi_{\frac{k}{2}}^{f,t_n+\frac{k}{2}}(W_{h,n}(k)) \nonumber \\
&=&\rho_{h,n+1}+\overline{W}_{h,n}(k)-W_{h,n}(k)+k
E(\overline{W}_{h,n}(k), W_{h,n}(k),k)+O(k^{p+1}), \label{eglobs}
\end{eqnarray}
where, for some constant $C$,
$$\|E(\overline{W}_{h,n}(k), W_{h,n}(k),k)\|_h \le C \|\overline{W}_{h,n}(k)- W_{h,n}(k))\|_h.$$
Now, notice that
\begin{eqnarray}
\overline{W}_{h,n}'(s)-W_{h,n}'(s)&=& A_{h,0}(\bar{W}_{h,n}(s)-W_{h,n}(s))+(\frac{k}{2}-s)[C_h
\partial f(t,u(t_n))-C_{h,n}^{*}(U_{h,n})],
 \nonumber \\
\overline{W}_{h,n}(0)-W_{h,n}(0)&=&\Psi_{\frac{k}{2}}^{ f,t_n}(P_h
u(t_n))- \Psi_{\frac{k}{2}}^{f,t_n}(U_{h,n}). \nonumber
\end{eqnarray}
Then,
\begin{eqnarray}
\overline{W}_{h,n}(k)-W_{h,n}(k)&=& e^{k A_{h,0}}\big(
\Psi_{\frac{k}{2}}^{f,t_n}(P_h
u(t_n))- \Psi_{\frac{k}{2}}^{f,t_n}(U_{h,n})\big) \nonumber \\
&&+k^2 (\frac{1}{2} \varphi_1(k A_{h,0})-\varphi_2(k A_{h,0}))[C_h \partial f(t_n,u(t_n))-C_{h,n}^{*}(U_{h,n})] \nonumber \\
&=&e^{k A_{h,0}}\big(\overline{V}_{h,n}(\frac{k}{2})-V_{h,n}(\frac{k}{2})+O(k^{p+1})\big) \nonumber \\
&&+k (\frac{1}{2} (e^{k A_{h,0}}+I)-\varphi_1(k A_{h,0}))A_{h,0}^{-1}[C_h \partial f(t_n,u(t_n))-C_{h,n}^{*}(U_{h,n})] \nonumber \\
&=&e^{k A_{h,0}}\big(P_h u(t_n)-U_{h,n}+k E(P_h u(t_n), U_{h,n},k)+O(k^{p+1})\big) \nonumber \\
&&+
k (\frac{1}{2} (e^{k A_{h,0}}+I)-\varphi_1(k A_{h,0}))A_{h,0}^{-1}[C_h \partial f(t_n,u(t_n))-C_{h,n}^{*}(U_{h,n})] \nonumber \\
&=&e^{k A_{h,0}} e_{h,n}+k \bar{E}(P_h
u(t_n),U_{h,n},k)+O(k^{p+1}),
\label{Wkrecur}
\end{eqnarray}
where $\overline{V}_{h,n}(s)$ is the solution of (\ref{Vhns})
starting at $P_h u(t_n)$ and
where the definition of $\varphi_1$ and $\varphi_2$ (\ref{varphi})
has been considered as well as (\ref{recurf}), (\ref{cnum}) and the fact that $\Psi_k$ integrates (\ref{strangsd-03}) with order $p$.
Moreover, for some constant $\bar{C}$,
\begin{eqnarray}
\|\bar{E}(P_h u(t_n), U_{h,n},k)\|_h \le \bar{C} \|e_{h,n}\|_h.
\label{cotaEs}
\end{eqnarray}
Then, inserting this in (\ref{eglobs}),
\begin{eqnarray}
e_{h,n+1}=e^{k A_{h,0}} e_{h,n}+ k \bar{\bar{E}}(P_h
u(t_n),U_{h,n},k)+\rho_{h,n+1}+O(k^{p+1}), \nonumber
\end{eqnarray}
where, for some constant $\bar{\bar{C}}$,
$$
\|\bar{\bar{E}}(P_h u(t_n), U_{h,n},k)\|_h \le \bar{\bar{C}}
\|e_{h,n}\|_h.
$$
Inductively, this means that
\begin{eqnarray}
e_{h,n}&=& e^{nk A_{h,0}} e_{h,0}+\sum_{l=1}^n e^{(n-l)k
A_{h,0}}\big(\rho_{h,l}+O(k^{p+1})\big)
 \nonumber \\&&
+k \sum_{l=0}^{n-1} e^{(n-l-1) k A_{h,0}} \bar{\bar{E}}(P_h
u(t_n), U_h^l,k). \label{induction}
\end{eqnarray}
Then, we write one of the terms in
(\ref{induction}) as in (\ref{decomp}).
As the first term in this decomposition can be written as
(\ref{ftd}), applying (\ref{spp}) and Theorem
\ref{theolocalerrorfulls}, this term is proved to be
$O(k^2+\eta_h)$. As for the second term in (\ref{decomp}), hypothesis (i) makes that the
term in $k^3$ in $A_{0}^{-1} \rho_{n+1}$ is differentiable with
respect to time $t_n$ and therefore,
$A_0^{-1}(\rho_{n-j+1}-\rho_{n-j})=O(k^4)$. When this is used in
the local error for the full discretization,
$A_{h,0}^{-1}(\rho_{h,n-j+1}-\rho_{h,n-j})=O(k^4+k^2 \eta_h)$,
from what
\begin{eqnarray}
\sum_{j=2}^{n-1} \big( \sum_{r=1}^{j-1} e^{r k A_{h,0}}\big)
(\rho_{h,n-j+1}-\rho_{h,n-j})&=&\sum_{j=2}^{n-1}
\big( k A_{h,0}\sum_{r=1}^{j-1} e^{r k A_{h,0}}\big)\frac{1}{k} A_{h,0}^{-1}(\rho_{h,n-j+1}-\rho_{h,n-j}) \nonumber \\
&=&O(k^2 +\eta_h). \nonumber
\end{eqnarray}
Finally, using Theorem \ref{theolocalerrorfulls} for the last term
in (\ref{decomp}), it is clear that
$$\sum_{l=1}^n e^{i k (n-l) A_{h,0}} \rho_{h,l}=O(k^2+k\varepsilon_h+\eta_h).$$
By taking then norms and considering that $e_{h,0}=0$,  we have that
$$\|e_{h,n}\|_h \le O(k^2+k\varepsilon_h+\eta_h)+k \bar{\bar{C}}\sum_{l=0}^{n-1} \|e_{h,l}\|_h .$$
Applying a discrete Gronwall lemma, the result follows.
\end{proof}

\section{Numerical experiments}
\label{experimentos} In this section we will show, through some
examples, that order reduction is completely avoided with the
technique suggested here for Lie-Trotter and Strang exponential
splitting methods. For the sake of brevity, we have restricted here to finite differences for the space discretization, although collocation-type methods
also satisfy hypotheses of Section \ref{sd}.

\subsection{One-dimensional problem}
Firstly, we consider (\ref{laibvp}) where $X=C([0,1])$ and $A$ is
the second-order space derivative. Moreover, we take
\begin{eqnarray}
u_0(x)=e^{x^3}, \quad f(t,u)=u^2-e^{t+x^3}(9 x^4+ 6 x
+e^{t+x^3}-1), \label{p1}
\end{eqnarray}
and for the Dirichlet boundary conditions,
\begin{equation} \label{p1bc}
 g_0(t)=e^{t}, \quad g_1(t)=e^{t+1},
 \end{equation}
so that the exact solution of the problem is
$$ u(x,t)=e^{t+x^3}. $$
For the space discretization, we take $h=1/(N+1)$ and we consider
the nodes $x_j=jh$, $j=0, \ldots, N+1$. Then, the discrete space
is $\mathbb{C}^N$, where $N$ is the number of interior nodes, and
the second derivative is approximated by means of  the standard
second-order difference scheme. Moreover, $P_h$ is the projection
on the interior nodal values, $A_{h,0}=\mbox{tridiag}(1,-2,1)/h^2$
and $C_h g(t)=[g_0(t) ,0 ,\dots ,0 ,g_1(t)]^T/h^2$. Since
$\mu_\infty(A_{h,0})=0$, hypothesis (H1a) is satisfied for the
maximum norm; on the other hand, (H1b) can be verified directly
from the formula of $A_{h,0}^{-1}$. Moreover, in this case (H2b)
is true with $Z=C^4([0,1])$ and $\varepsilon_h, \eta_h=O(h^2)$,
and $f$ satisfies (H3).

Calculating $\varphi_j(k A_{h,0})C_h g(t)$ just corresponds to
making a linear combination of the first and last column of
$\varphi_j(k A_{h,0})$, which can be both calculated once and for
all at the very beginning for fixed stepsize $k$. As integrator
$\Psi_k$, we have considered the $4$th-order classical Runge-Kutta
method.

Considering the technique suggested in this paper, we have
obtained the results in Table 1 when integrating till time
$T=0.2$ with Lie-Trotter method and $h=10^{-3}$ and
those in Table 2 with Strang method and $h=2.5 \times 10^{-4}$. It
is clear that orders $2$ and $1$ are obtained for the local and
global errors respectively when integrating with Lie-Trotter and
order $2$ for the local and global errors when integrating with
Strang method, as assured by Theorems \ref{theolocalerrorfull},
\ref{theoglobalerrorfull}, \ref{theolocalerrorfulls} and
\ref{theoglobalerrorfulls} when the error in space is negligible.
(This seems to be the case because decreasing $h$ does not practically change the errors.)
\begin{table}[t]
\begin{center}
\begin{tabular}{|c|c|c|c|}
\hline & $k=5 \times 10^{-4}$ &  $k=2.5 \times 10^{-4}$ &  $k=1.25
\times 10^{-4}$   \\ \hline
$L^\infty$-local error & 1.5838e-04  & 4.2830e-05 &   1.1390e-05  \\ \hline
Order & & 1.8867  & 1.9108  \\
\hline $L^{\infty}$-global error  & 6.8139e-03  & 3.4035e-03 &  1.7016e-03
  \\ \hline
Order & & 1.0015 &  1.0001\\ \hline
\end{tabular}
\end{center}
\label{t1} \caption{Local and global error when integrating the
one-dimensional problem  corresponding to data (\ref{p1}) and
(\ref{p1bc}) with Dirichlet boundary conditions with the suggested
modification of Lie-Trotter method}
\end{table}
\begin{table}[t]
\begin{center}
\begin{tabular}{|c|c|c|c|c|}
\hline & $k=1 \times 10^{-3}$ &  $k=5 \times 10^{-4}$ &  $k=2.5
\times 10^{-4}$  \\ \hline
$L^{\infty}$-local error & 8.5559e-05 &  2.1777e-05  & 5.5000e-06
 \\ \hline Order & & 1.9741 &  1.9853  \\
\hline $L^{\infty}$-global error & 1.6140e-04 &  4.2882e-05 &  1.1235e-05
 \\ \hline Order & & 1.9122    &   1.9324
\\ \hline
\end{tabular}
\end{center}
\label{t2} \caption{Local and global error when integrating the
one-dimensional problem  corresponding to data (\ref{p1}) and
(\ref{p1bc}) with Dirichlet boundary conditions with the suggested
modification of Strang method}
\end{table}

Let us now consider the same problem as for the previous
experiment, but with a Neumann boundary condition at the right
boundary. More precisely, the boundary conditions are
\begin{eqnarray}\label{Neumann}
 u(0,t)&=& g_0(t), \\
 u_x (1,t) &=& g_1(t). \nonumber
\end{eqnarray}
with $g_0(t)=e^t$ and $g_1(t)=3 e^{1+t}$.

In this case, the values in the node $x=1$ are included and the
discrete space is $\mathbb{C}^{N+1}$, the matrix $A_{h,0}$ is the
same as in the previous experiment except for the last row which
is  $[0,\ldots, 0, 2,-2]/h^2$ now, and $C_h
\partial u(t) =[g_0(t)/h^2, 0 , \ldots, 0, 2g_1(t)/h]^T$. Again, hypotheses (H1)-(H3)
are satisfied with $Z=C^4([0,1])$.

We notice that now, as it is can be proved by using Taylor
expansions, all components of $A_{h,0}(R_h u-P_hu)$ are
$O(h^2\|u_{xxxx}\|_\infty)$ except for the last component which is
just $O(h\|u_{xxxx}\|_\infty)$. Therefore, $\varepsilon_h$ is
$O(h)$ and $\eta_h$ is, in principle, also $O(h)$  for every $u
\in C^4([0,1])$. However,
\begin{eqnarray}
\label{formulon} R_hu-P_hu&=&A_{h,0}^{-1}\left[\begin{array}{c}
O(h^2\|u_{xxxx}\|_\infty)\\ \vdots \\
O(h^2\|u_{xxxx}\|_\infty)\\0\end{array}\right]+
A_{h,0}^{-1}\left[\begin{array}{c}
0\\ \vdots \\
0\\O(h\|u_{xxxx}\|_\infty)\end{array}\right]\nonumber\\
&=& O(h^2)\|u_{xxxx}\|_\infty+ O(h^2)\|u_{xxxx}\|_\infty
A_{h,0}^{-1} \left[\begin{array}{c}
0\\ \vdots \\
0\\2/h\end{array}\right]
\end{eqnarray}
where, for the last equality, we have used (H1b). Taking now into
account that, when discretizing the problem $Av=0, v(0)=0,
v_x(1)=1$, it follows that $A_{h,0}R_hv+ [0,\ldots, 0,
2/h]^T=[0,\ldots, 0, 0]^T$, and it happens that
$$A_{h,0}^{-1}[0,\ldots, 0, 2/h]^T=-R_hv=-P_hv+O(h).$$
Therefore, the term above is bounded for small enough $h$ because
$v(x)=x$ and, from (\ref{formulon}), $\eta_h$ is in fact
$O(h^2)$.

The results which are obtained with the technique proposed in this
paper are shown in Table 3 for Lie-Trotter with $h=10^{-3}$ and in
Table 4 for Strang with $h=2.5 \times 10^{-4}$. In both cases, the
global errors are measured at time $T=0.2$. We see that, for
Lie-Trotter, orders 2 and 1 are observed  for the local and global
error respectively when $k$ decreases, while for Strang, order 2 is obtained for both
the local and global errors. These results corroborate Theorems
\ref{theolocalerrorfull}, \ref{theoglobalerrorfull2},
\ref{theolocalerrorfulls} and \ref{theoglobalerrorfulls} (We also
notice that hypotheses (\ref{cnum}) and (\ref{spp}) apply here
because of the use of Remark \ref{r1} to approximate $\partial
f(t_n,u(t_n))$ and the parabolic character of the equation).

\begin{table}[t]
\begin{center}
\begin{tabular}{|c|c|c|c|}
\hline & $k=5 \times 10^{-4}$ &  $k=2.5 \times 10^{-4}$ &  $k=1.25
\times 10^{-4}$  \\ \hline
$L^\infty$-local error & 2.0286e-04  & 5.1444e-05  & 1.2795e-05  \\ \hline
Order & & 1.9794  & 2.0074   \\
\hline $L^{\infty}$-global error & 3.9872e-02 &  1.9887e-02 &
9.9237e-03    \\ \hline Order & & 1.0036  & 1.0029  \\ \hline
\end{tabular}
\end{center}
\label{t2b} \caption{Local and global error when integrating the
one-dimensional problem  corresponding to data (\ref{p1}), with
Dirichlet and Neumann boundary conditions (\ref{Neumann}), with
the suggested modification of Lie-Trotter method}
\end{table}
\begin{table}[t]
\begin{center}
\begin{tabular}{|c|c|c|c|}
\hline & $k=1 \times 10^{-3}$ &  $k=5 \times 10^{-4}$ &  $k=2.5
\times 10^{-4}$  \\ \hline
$L^{\infty}$-local error & 2.6922e-05  & 5.0772e-06  & 9.1626e-07
\\ \hline Order & & 2.4067    &   2.4702     \\
\hline $L^{\infty}$-global error & 1.8549e-04 &  4.6220e-05 &
1.0814e-05  \\ \hline Order & & 2.0048  &     2.0957
 \\ \hline
\end{tabular}
\end{center}
\label{t2c} \caption{Local and global error when integrating the
one-dimensional problem  corresponding to data (\ref{p1}), with
Dirichlet and Neumann boundary conditions (\ref{Neumann}), with
the suggested modification of Strang method}
\end{table}

\subsection{Two-dimensional problem}We have also
considered the two-dimensional problem in the square
$\Omega=[0,1]\times [0,1]$ corresponding to the Laplacian as
operator $A$. Moreover, we have considered
\begin{eqnarray}
&&u_0(x) = e^{x^3+y^3}, \quad x \in  \Omega, \quad g(t,x)=e^{t+x^3+y^3}, \quad x \in \partial \Omega, \nonumber \\
&&f(t,u)= u^2 - e^{t+x^3+y^3}(9(x^4+y^4)+6(x+y)+e^{t+x^3+y^3}-1),  \quad x \in  \Omega,
 \label{p2}
\end{eqnarray}
which has $u(t,x)=e^{t+x^3+y^3}$ as exact solution.

Firstly, for the discretization of the Laplacian we have
considered the standard five-point formula \cite{S}. We notice
that, in this case, the discrete space is $\mathbb{C}^{N^2}$,
where $N$ is the number of interior nodes in each direction, and
$A_{h,0}$ is a tridiagonal block-matrix of dimension $N^2$.
Besides, the matrices in the diagonal are the same and are
tridiagonal and the matrices at the subdiagonal and superdiagonal
are the same and are diagonal. Notice also that $C_h g(t)$ would
just have $4N-4$ non-vanishing components, which is a number which
is negligible compared with $N^2$, the total number of interior
nodes. Again, (H1)-(H3) are satisfied for the infinity norm with
$\varepsilon_h$, $\eta_h$ being $O(h^2)$.

Tables 5 and 6 show the orders which are observed in time for $h=10^{-2}$
when integrating the problem till time $T=1$ with the suggested modifications of
Lie-Trotter and Strang method considering again $\Psi_k$ as the
fourth-order classical Runge-Kutta method. Again, we see that the
local and global order for Lie-Trotter are near $2$ and $1$
respectively and that the local and global order for Strang are
near $2$.

\begin{table}[t]
\begin{center}
\begin{tabular}{|c|c|c|c|}
\hline
 &  $k=5 \times 10^{-3}$ &  $k=2.5 \times 10^{-3}$ &   $k=1.25 \times 10^{-3}$  \\ \hline
$L^\infty$-local error &  6.1550e-02 &  1.9049e-02 &  5.7445e-03  \\ \hline
Order & &    1.6921 &  1.7295  \\
\hline $L^{\infty}$-global error   &  6.1666e-01 &  2.9307e-01  &
1.4341e-01  \\ \hline Order & &    1.0732 &  1.0311  \\ \hline
\end{tabular}
\end{center}
\label{t3} \caption{Local and global error when integrating the
two-dimensional problem  corresponding to data (\ref{p2}) with the
suggested modification of Lie-Trotter method}
\end{table}
\begin{table}[t]
\begin{center}
\begin{tabular}{|c|c|c|c|}
\hline & $k=1 \times 10^{-2}$ &  $k=5 \times 10^{-3}$ &  $k=2.5
\times 10^{-3}$   \\ \hline
$L^{\infty}$-local error & 5.4180e-02 &  1.5992e-02 &   4.5641e-03
 \\ \hline Order & &  1.7604 &  1.8090 \\
\hline $L^{\infty}$-global error & 3.0713e-01  & 7.7562e-02 &
2.1856e-02  \\ \hline Order & & 1.9854  & 1.8273  \\ \hline
\end{tabular}
\end{center}
\label{tab4} \caption{Local and global error when integrating the
two-dimensional problem  corresponding to data (\ref{p2}) with the
suggested modification of Strang method}
\end{table}

We also consider the use of a double splitting in order to obtain
a very efficient time integrator. For that, we have followed the
lines in \cite{ACR}, where the order reduction is avoided in the
case of a dimension splitting of the Laplacian operator. We use
similar ideas for the discretization of (\ref{vn}) in Lie-Trotter
method and that of (\ref{strangsd-02}) in Strang method.

More precisely, for the discretization of (\ref{vn}) we firstly
consider the problem
\begin{eqnarray}
 z_n'(s) &=& A_1 z_n (s), \label{zn} \\
 z_n(0)&=& u_n, \nonumber \\
 \partial_1 z_n(s) &=& \partial (u(t_n) + s A_1 u(t_n) ), \nonumber
\end{eqnarray}
with $A_1u= \partial_{xx}u$, and $\partial_1 u =\{u(0, y) = u(1,
y), y \in [0,1]\}$. Then,
\begin{eqnarray}
 r_n'(s) &=& A_2 r_n (s) ,\label{rn} \\
 r_n(0)&=& z_n(k), \nonumber \\
 \partial r_n(s) &=& \partial_2 (u(t_n) + k A_1 u(t_n) + s A_2 u(t_n)), \nonumber
\end{eqnarray}
where $A_2u= \partial_{yy}u$ and and $\partial_2 u=\{u(x, 0) =
u(x,1 ), x \in [0,1]\}$. Finally, we make $v_n(k)=r_n(k)$. We
apply now the spatial discretization of problems (\ref{zn}) and
(\ref{rn}) and we obtain
\begin{eqnarray*}
Z_{h,n}(k) &=& e^{k A_{h,0,1}} U_{h,n} + k \varphi_1( k A_{h,0,1}) C_h \partial_1 u(t_n) + k^2 \varphi_2 ( k A_{h,0,1}) C_h \partial_1 A_1 u(t_n),\\
R_{h,n}(k) &=& e^{k A_{h,0,2}} U_{h,n} + k \varphi_1( k A_{h,0,2}) C_h \partial_2 (u(t_n) + k A_1 u(t_n))\\
& &+ k^2 \varphi_2 ( k A_{h,0,2}) C_h \partial_2 A_2 u(t_n),
\end{eqnarray*}
where $A_{h,0,1}, A_{h,0,2}$, $\partial_1, \partial_2$ are the
matrices and the boundaries associated to the  spatial
discretization of $A_1$ and $A_2$ respectively.

On the other hand, in order to integrate (\ref{strangsd-02}) with
Strang method, we firstly consider the problem,
\begin{eqnarray}
 r_n'(s) &=& A_1 r_n(s) ,\label{Strangrn} \\
 r_n(0)&=& v_n \left(\frac{k}{2}\right), \nonumber \\
 \partial_1 r_n (s) &=& \partial_1 \left( u(t_n) + \frac{k}{2} f(t_n, u(t_n)) + s A_1 u(t_n) \right), \nonumber
\end{eqnarray}
then,
\begin{eqnarray}
 \phi_n'(s) &=& A_2 \phi_n(s) \label{Strangphin}, \\
 \phi_n(0)&=& r_n \left(\frac{k}{2}\right) ,\nonumber \\
 \partial_2 \phi_n (s) &=& \partial_2 \left( u(t_n) + \frac{k}{2} f(t_n, u(t_n)) + \frac{k}{2} A_1 u(t_n) + s A_2 u(t_n) \right), \nonumber
\end{eqnarray}
and finally
\begin{eqnarray}
 \mu_n'(s) &=& A_1 \mu_n(s), \label{Strangmun} \\
 \mu_n(0)&=& \phi_n (k), \nonumber \\
 \partial_1 \mu_n (s) &=& \partial_1 \left( u(t_n) + \frac{k}{2} f(t_n, u(t_n)) + \frac{k}{2} A_1 u(t_n) + k A_2 u(t_n) + s A_1 u(t_n) \right), \nonumber
\end{eqnarray}
and we make $w_n(k) = \mu_n (\frac{k}{2})$. Considering now the
spatial discretization of the previous three problems, we obtain
\begin{eqnarray*}
 R_{h,n}\left(\frac{k}{2}\right) &=& e^{\frac{k}{2}A_{h,0,1}} V_{h,n} + \frac{k}{2} \varphi_1 \left(\frac{k}{2} A_{h,0,1}\right) C_h \partial_1 \left(u(t_n) + \frac{k}{2} f(t_n,u(t_n)) \right) \\
 & &+ \frac{k^2}{4} \varphi_2\left( \frac{k}{2}A_{h,0,1}\right) C_h \partial_1 A_1 u(t_n), \\
  \Phi_{h,n}(k) &=& e^{kA_{h,0,2}} R_{h,n}\left( \frac{k}{2} \right) + k \varphi_1 (k A_{h,0,2}) C_h \partial_2 \left(u(t_n) + \frac{k}{2} f(t_n,u(t_n)) + \frac{k}{2} A_1 u(t_n) \right) \\
 & &+ k^2 \varphi_2( kA_{h,0,2}) C_h \partial_2 A_2 u(t_n),\\
 \mu_{h,n}\left(\frac{k}{2}\right) &=& e^{\frac{k}{2}A_{h,0,1}} \Phi_{h,n} (k) \\
 & &+ \frac{k}{2} \varphi_1 \left(\frac{k}{2} A_{h,0,1}\right) C_h \partial_1 \left(u(t_n) + \frac{k}{2} f(t_n,u(t_n)) + \frac{k}{2} A_1 u(t_n)  + k A_2 u(t_n)\right) \\
 & &+ \frac{k^2}{4} \varphi_2\left( \frac{k}{2}A_{h,0,1}\right) C_h \partial_1 A_1 u(t_n),\\
 W_{h,n} &=& \mu_{h,n} \left(\frac{k}{2}\right) .
\end{eqnarray*}

We have considered the standard second order symmetric finite
difference scheme for the discretization of $A_1$ and $A_2$.
Notice that this procedure will be especially efficient since now
the matrices $A_{h,0,j}$ ($j=1,2$) for space discretization in one
or another direction are block-diagonal matrices and, moreover,
the blocks are tridiagonal. Therefore, multiplying $e^{k
A_{h,0,j}}$ or $\varphi_l( k A_{h,0,j})$ times a vector of size
$N^2$ just corresponds to $N$ products of a matrix of dimension $N
\times N$ times a vector of size $N$. Moreover, many components of
$C_h g(t)$ will vanish.

Although we do not make the analysis for this double splitting, it
is natural to suspect that order reduction is also being
completely avoided. Tables 7 and 8 corroborate that behavior with
$h=10^{-2}$.
\begin{table}[t]
\begin{center}
\begin{tabular}{|c|c|c|c|}
\hline
 &  $k=5 \times 10^{-3}$ &  $k=2.5 \times 10^{-3}$ &   $k=1.25 \times 10^{-3}$  \\ \hline
$L^\infty$-local error & 6.9693e-02 &  1.9980e-02 &  5.8275e-03 \\
\hline
Order & &    1.8025 &  1.7776  \\
\hline $L^{\infty}$-global error   &  6.1373e-01 &  2.9240e-01 &
1.4325e-01  \\ \hline Order & &    1.0697  & 1.0294  \\
\hline
\end{tabular}
\end{center}
\label{tab5} \caption{Local and global error when integrating the
two-dimensional problem  corresponding to data (\ref{p2}) with the
double splitting of Lie-Trotter method and second-order difference
scheme in space}
\end{table}
\begin{table}[t]
\begin{center}
\begin{tabular}{|c|c|c|c|}
\hline
 &  $k=10^{-2}$ &  $k=5 \times 10^{-3}$ &   $k=2.5 \times 10^{-3}$ \\ \hline
$L^\infty$-local error & 6.5879e-02 &  1.8698e-02 &  5.2568e-03  \\ \hline Order & &    1.8169 &  1.8307  \\
\hline $L^{\infty}$-global error   &  3.4096e-01 &  8.7881e-02 &
2.3635e-02  \\ \hline Order & &    1.9560  & 1.8946   \\
\hline
\end{tabular}
\end{center}
\label{t6} \caption{Local and global error when integrating the
two-dimensional problem  corresponding to data (\ref{p2}) with the
double splitting of Strang method and second-order difference
scheme in space}

\end{table}

\section*{Acknowledgements}
This research has been supported by Ministerio de Ciencia e
Innovaci\'on project MTM2015-66837-P of Ministerio de
Econom\'{\i}a y Competitividad.

\end{document}